\documentclass[11pt]{amsart}
\usepackage{mathrsfs,eucal,amsmath,amsthm,amsfonts,amssymb,amscd,amsbsy,latexsym,dsfont,stmaryrd,mathtools}

\usepackage[all,2cell]{xy} \UseAllTwocells \SilentMatrices
\usepackage[T1]{fontenc}

\begin{document}

\newcommand{\INVISIBLE}[1]{}

\newtheorem{thm}{Theorem}[section]
\newtheorem{lem}[thm]{Lemma}
\newtheorem{cor}[thm]{Corollary}
\newtheorem{prp}[thm]{Proposition}
\newtheorem{conj}[thm]{Conjecture}

\theoremstyle{definition}
\newtheorem{dfn}[thm]{Definition}
\newtheorem{question}[thm]{Question}
\newtheorem{nota}[thm]{Notations}
\newtheorem{notation}[thm]{Notation}
\newtheorem*{claim*}{Claim}
\newtheorem{ex}[thm]{Example}
\newtheorem{rmk}[thm]{Remark}
\newtheorem{rmks}[thm]{Remarks}
 
\def\labelenumi{(\arabic{enumi})}

\newcommand{\aro}{\longrightarrow}
\newcommand{\arou}[1]{\stackrel{#1}{\longrightarrow}}
\newcommand{\RA}{\Longrightarrow}

\newcommand{\mm}[1]{\mathrm{#1}}
\newcommand{\bm}[1]{\boldsymbol{#1}}
\newcommand{\bb}[1]{\mathbf{#1}}

\newcommand{\bA}{\boldsymbol A}
\newcommand{\bB}{\boldsymbol B}
\newcommand{\bC}{\boldsymbol C}
\newcommand{\bD}{\boldsymbol D}
\newcommand{\bE}{\boldsymbol E}
\newcommand{\bF}{\boldsymbol F}
\newcommand{\bG}{\boldsymbol G}
\newcommand{\bH}{\boldsymbol H}
\newcommand{\bI}{\boldsymbol I}
\newcommand{\bJ}{\boldsymbol J}
\newcommand{\bK}{\boldsymbol K}
\newcommand{\bL}{\boldsymbol L}
\newcommand{\bM}{\boldsymbol M}
\newcommand{\bN}{\boldsymbol N}
\newcommand{\bO}{\boldsymbol O}
\newcommand{\bP}{\boldsymbol P}
\newcommand{\bY}{\boldsymbol Y}
\newcommand{\bS}{\boldsymbol S}
\newcommand{\bX}{\boldsymbol X}
\newcommand{\bZ}{\boldsymbol Z}

\newcommand{\cc}[1]{\mathcal{#1}}
\newcommand{\ccc}[1]{\mathscr{#1}}

\newcommand{\ca}{\cc{A}}

\newcommand{\cb}{\cc{B}}

\newcommand{\cC}{\cc{C}}

\newcommand{\cd}{\cc{D}}

\newcommand{\ce}{\cc{E}}

\newcommand{\cf}{\cc{F}}

\newcommand{\cg}{\cc{G}}

\newcommand{\ch}{\cc{H}}

\newcommand{\ci}{\cc{I}}

\newcommand{\cj}{\cc{J}}

\newcommand{\ck}{\cc{K}}

\newcommand{\cl}{\cc{L}}

\newcommand{\cm}{\cc{M}}

\newcommand{\cn}{\cc{N}}

\newcommand{\co}{\cc{O}}

\newcommand{\cp}{\cc{P}}

\newcommand{\cq}{\cc{Q}}

\newcommand{\cR}{\cc{R}}

\newcommand{\cs}{\cc{S}}

\newcommand{\ct}{\cc{T}}

\newcommand{\cu}{\cc{U}}

\newcommand{\cv}{\cc{V}}

\newcommand{\cy}{\cc{Y}}

\newcommand{\cw}{\cc{W}}

\newcommand{\cz}{\cc{Z}}

\newcommand{\cx}{\cc{X}}

\newcommand{\g}[1]{\mathfrak{#1}}

\newcommand{\af}{\mathds{A}}
\newcommand{\PP}{\mathds{P}}

\newcommand{\GL}{\mathrm{GL}}
\newcommand{\PGL}{\mathrm{PGL}}
\newcommand{\SL}{\mathrm{SL}}
\newcommand{\NN}{\mathbf{N}}
\newcommand{\ZZ}{\mathbf{Z}}
\newcommand{\CC}{\mathbf{C}}
\newcommand{\QQ}{\mathbf{Q}}
\newcommand{\RR}{\mathds{R}}
\newcommand{\FF}{\mathds{F}}
\newcommand{\DD}{\mathbf{D}}
\newcommand{\VV}{\mathds{V}}
\newcommand{\HH}{\mathds{H}}
\newcommand{\MM}{\mathds{M}}
\newcommand{\OO}{\mathds{O}}
\newcommand{\LL}{\mathds L}
\newcommand{\BB}{\mathds B}
\newcommand{\kk}{\mathds k}
\newcommand{\bs}{\mathbf S}
\newcommand{\GG}{\mathbf G}

\newcommand{\al}{\alpha}

\newcommand{\be}{\beta}

\newcommand{\ga}{\gamma}
\newcommand{\Ga}{\Gamma}

\newcommand{\om}{\omega}
\newcommand{\Om}{\Omega}

\newcommand{\vte}{\vartheta}
\newcommand{\te}{\theta}
\newcommand{\Te}{\Theta}

\newcommand{\ph}{\varphi}
\newcommand{\Ph}{\Phi}

\newcommand{\ps}{\psi}
\newcommand{\Ps}{\Psi}

\newcommand{\ep}{\varepsilon}

\newcommand{\vr}{\varrho}

\newcommand{\de}{\delta}
\newcommand{\De}{\Delta}

\newcommand{\la}{\lambda}
\newcommand{\La}{\Lambda}

\newcommand{\ka}{\kappa}

\newcommand{\si}{\sigma}
\newcommand{\Si}{\Sigma}

\newcommand{\ze}{\zeta}

\newcommand{\fr}[2]{\frac{#1}{#2}}
\newcommand{\vs}{\vspace{0.3cm}}
\newcommand{\na}{\nabla}
\newcommand{\pd}{\partial}
\newcommand{\po}{\cdot}
\newcommand{\met}[2]{\left\langle #1, #2 \right\rangle}
\newcommand{\rep}[2]{\mathrm{Rep}_{#1}(#2)}
\newcommand{\repp}[2]{\mathrm{Rep}^\circ_{#1}(#2)}
\newcommand{\hh}[3]{\mathrm{Hom}_{#1}(#2,#3)}
\newcommand{\modules}[1]{#1\text{-}\mathbf{mod}}
\newcommand{\Modules}[1]{#1\text{-}\mathbf{Mod}}
\newcommand{\dmod}[1]{\mathcal{D}(#1)\text{-}{\bf mod}}

\newcommand{\an}{\mathrm{an}}

\newcommand{\pos}[2]{#1\llbracket#2\rrbracket}

\newcommand{\cpos}[2]{#1\langle#2\rangle}

\newcommand{\id}{\mathrm{id}}

\newcommand{\one}{\mathds 1}

\newcommand{\ti}{\times}
\newcommand{\tiu}[1]{\underset{#1}{\times}}

\newcommand{\ot}{\otimes}
\newcommand{\otu}[1]{\underset{#1}{\otimes}}

\newcommand{\wh}{\widehat}
\newcommand{\wt}{\widetilde}
\newcommand{\ov}[1]{\overline{#1}}
\newcommand{\un}[1]{\underline{#1}}

\newcommand{\op}{\oplus}

\newcommand{\lid}{\varinjlim}
\newcommand{\lip}{\varprojlim}

\newcommand{\ega}[3]{[EGA $\mathrm{#1}_{\mathrm{#2}}$, #3]}

\title[Group schemes over discrete valuation rings, II]{On the structure of affine flat group schemes over discrete valuation rings, II }

\author[P. H. Hai]{Ph\`ung H\^o Hai}

\address{Institute of Mathematics, Vietnam Academy of Science and Technology, Hanoi, 
Vietnam}

\email{phung@math.ac.vn}
\thanks{The research of Ph\`ung H\^o Hai is funded by the Vietnam National Foundation for Science and Technology Development under grant number    101.04-2019.315.
Parts of this work were developed while J. P. dos Santos was temporarily assigned to the CNRS}

\author[J. P. dos Santos]{Jo\~ao Pedro  dos Santos}

\address{Institut de Math\'ematiques de Jussieu -- Paris Rive Gauche, 4 place Jussieu, 
Case 247, 75252 Paris Cedex 5, France}

\email{joao\_pedro.dos\_santos@yahoo.com}



\date{Version 7, 27 August 2020.}

\begin{abstract}
In the first part of this work \cite{DHdS15}, we studied affine group schemes over a discrete valuation ring by means of Neron blowups. We also showed how to apply these findings to throw light on the  group schemes coming from  Tannakian categories of $\cd$-modules. In the present work we follow up   this theme. We show that a certain class of  affine group schemes  of ``infinite type'', Neron blowups of formal subgroups, are quite typical. We also explain how these group schemes appear naturally in  Tannakian categories of $\cd$-modules.  To conclude we isolate a Tannakian property of affine group schemes, named prudence, which allows one to verify if the underlying  ring of functions  is a  free module  over the base ring. This is then successfully applied to obtain a general result on the structure of differential Galois groups over complete discrete valuation rings.
\end{abstract}
\maketitle

\section{Introduction}

In this paper we continue the analysis of  {\it affine flat group schemes} over a discrete valuation ring (DVR) $R$  started in \cite{DHdS15} and use it to derive results in  {\it differential Galois theory}. 
Recall that the main idea behind the latter theory is to attach to a linear differential  systems, or  $\cd$-module, a linear algebraic group and to study these   concurrently. Now, when dealing with objects having an ``arithmetic'' input---like differential systems with coefficients in $R$---the group theoretical side of the picture is better grasped with the help of  Tannakian categories and flat group schemes which are {\it not} necessarily of finite type over $R$. (A very simple example is below and details are in \cite[Section 7]{DHdS15}. See also Section \ref{22.08.2020--1}.) Indeed, for a smooth   $R$-scheme $X$ with geometrically connected fibres and a point $x_0\in X(R)$, and a $\cd$--module $\cm$ on $X$, the category of representations of the (full) differential Galois group is ``generated'' by $\cm$ but, contrary to the case of a base field, this   property does not assure the existence of an immersion into $\bb{GL}(x_0^*\cm)$. 
It then becomes relevant to throw light on general properties of arbitrary affine   flat  group schemes.  
Some of these properties were spotted in  \cite{DH14}  and \cite{DHdS15} and further meditation conducted to answers which we write down here. To focus our attention, we shall raise three questions---named ({\bf UB}), ({\bf SA}) and ({\bf PR}) below---and structure the paper around them. 

Let $R=\pos k\pi$, $k$ of characteristic zero,   and consider the $\cd$--module on ${\rm Spec}\, R[x]$ associated to the differential equation $y'(x)=\pi y(x)$.
Example 7.11 in \cite{DHdS15}, coming  from \cite{Andre},  shows that the  associated differential Galois group fails to be of finite type. In addition, and more importantly,   this failure is related to a simple construction---the Neron blowup of a formal subgroup---performed on  $\bb{G}_{m,R}$.  Assume now only that $R$ is complete with uniformizer $\pi$. The Neron blowup   produces,  starting from an affine flat group scheme of finite type $G$  and a flat closed formal subgroup $\g H$ of the completion $\wh G$, 
a new group   $\cn_{\g H}^\infty(G)$ having the same generic fibre as $G$. More precisely, yet still loosely speaking,  the ring of $\cn_{\g H}^\infty(G)$ is obtained from the ring of $G$ by adjoining rational functions of the form 
\[
\fr{\text{element of the ideal of $\g H\ot R/(\pi^{n})$}}{\pi^n}
\]
for varying $n$. See Section \ref{23.02.2017--3} and \cite[Section 5]{DHdS15}. It is then natural to  raise the following question.  

\vspace{.2cm}
\noindent(\textbf{UB}) ``Ubiquity of ``blowups'':  Assume that $R$ is complete and let $G$ be an affine and flat group scheme over $R$ whose generic fibre $G\ot K$ is of finite type. Is it the case that $G$ is the blowup of a formal subgroup   of some group scheme \emph{of finite type} over $R$? 
  
This is addressed in Section \ref{06.02.2017--1}, where we show that, if   $R$ in  ({\bf UB}) has equicharacteristic zero, then the answer is affirmative; see Corollary \ref{11.07.2017--2}. 
In fact, the particularity which gives these blowups such a prominent role is not exclusive of characteristic $(0,0)$, and has to do, as visible in Theorem \ref{20.02.2017--2}, with the ``constancy'' of the centres in the standard sequence (recalled in Section \ref{20.02.2017--1}). This property had already been detected in \cite[Section 6]{DHdS15}, and used to arrive at a less clear result than Theorem \ref{20.02.2017--2}.

Now, if Theorem \ref{20.02.2017--2} and Corollary \ref{11.07.2017--2} say that Neron blowups of formal subgroups are typical among affine and flat group schemes, nothing assures their relevance in {\it all} Tannakian theories over $R$. The task of detecting their ``presence'' in {\it simple instances}---we have in mind categories of representations of abstract groups and $\cd$-modules on {\it proper} schemes---comes thus to the front. (Here we understand ``presence'', at least in the case of $\cd$--modules, as ``being a full differential Galois group''. Note also  that the equation $y'=\pi y$ above assures that   
Neron blowups of formal subgroups are present in the category of $\cd$-modules on {\it affine} schemes.)
As argued below, this task is closely linked to a problem  raised in \cite[Question 4.1.4]{DH14} and a fundamental result concerning affine group schemes over a field \cite[3.3 and 14.1]{waterhouse}. We therefore prefer  to highlight  this point:

\vspace{.2cm}
\noindent  (\textbf{SA}) ``Strict pro-algebraicity'': Let $G$ be an affine and flat group scheme over $R$; is it possible to write $G$ as $\lip_iG_i$, where each  $G_j\to G_i$ is a \emph{faithfully flat} morphism of affine and  flat  group schemes \emph{of finite type}?

As hinted above, Neron blowups of formal subgroups provide   negative answers to ({\bf SA}): if $G=\lip_iG_i$ is as in ({\bf SA}), then each faithfully flat quotient $G\to H$ having a generic fibre of finite type must be noetherian, a property which these blowups can easily violate. (The proof of Corollary \ref{12.10.2018--2} contains details of the last argument.)
With Corollary \ref{12.10.2018--2} and Corollary \ref{23.02.2017--6} we 
show that, in case $R=\pos \CC t$, a certain $\cn_{\g H}^\infty(\GG_{a,R}\ti\GG_{m,R})$ appears  both in the theory of   representations of abstract groups ($\ZZ$ actually!)  and in the theory  of $\cd$-modules on {\it proper} schemes. This has the double effect of  extending the number of negative answers to ({\bf SA}) and increasing the relevance of blowing up formal subgroups. 

Finally, we address a   basic property of affine group schemes. 

\vspace{.2cm}
\noindent  (\textbf{PR})  ``Projectivity'': Let $G$ be an affine and flat group scheme over $R$. What conditions  should the category of representations $\rep{R}{G}$ enjoy so that $R[G]$ is projective as an $R$-module?

That projectivity as a module is indeed a basic property of group schemes and not simply a commutative algebraic one should be made clearer by pointing out its relation to the existence of  arbitrary  intersections in representations (Proposition \ref{16.10.2018--2}).
The analysis of ({\bf PR}) is the theme of Section \ref{07.07.2017--1} and the answer we offer, for a complete $R$, is named {\it prudence}, see Definition \ref{17.10.2016--4} and Theorem 
\ref{17.10.2016--6}. 
 When the category of representations in question is related to a category of modules on a scheme, the property of being prudent is linked to Grothendieck's existence theorem in formal geometry.  This permits us to apply our findings to verify that  differential Galois groups on {\it proper} $R$-schemes have rings of functions which are free $R$-modules, see Theorem \ref{prudence_galois}.

We end this introduction by summarizing the remaining sections of the paper while throwing light on results which serve to complement or support the ones described above.  
Section \ref{29.03.2017--1} exists mainly to establish notations and explain colloquially the operation of blowing up a formal subgroup scheme plus the concept of differential Galois group. 
Section \ref{06.02.2017--1}  addresses Question ({\bf UB}) and proves Theorem \ref{20.02.2017--2} and   Corollary \ref{11.07.2017--2}. In it the reader shall also find applications of Theorem \ref{20.02.2017--2} to unipotent group schemes in equicharacteristic $(0,0)$;   this particular situation allows to show that formal subgroups are algebraizable (Proposition \ref{11.07.2017--1}), something which will be put to good use once in possession of Theorem \ref{prudence_galois}, see Corollary \ref{16.10.2018--1}. 
Section \ref{abstract_groups} develops carefully the theory of the Tannakian envelope over $R$ of an abstract group (Definition \ref{22.02.2017--3}) and demonstrates Corollary \ref{12.10.2018--2}. 
Section \ref{24.02.2017--1} serves to show that on a smooth and proper $\pos\CC t$-scheme $X$, the theory of $\cd$-modules is tantamount to that of the representations of $\pi_1((X\ot \CC)^\an)$, see Theorem \ref{Riemann-Hilbert}. 
This result is then applied to establish Corollary \ref{23.02.2017--6}, highlighted above. 
Section \ref{07.07.2017--1} develops the concept of prudence (Definition \ref{17.10.2016--4}) to deal with ({\bf PR}), as already mentioned before. 
Other noteworthy points elaborated in Section \ref{07.07.2017--1} are a link between prudence and  a   property of representations appearing in \cite[Expos\'e ${\rm VI}_B$]{SGA3}---see Proposition \ref{16.10.2018--2}---, 
and, building up on material from Section \ref{06.02.2017--1}, a finiteness result for differential Galois groups with unipotent generic fibre, see Corollary \ref{16.10.2018--1}. 

\subsection*{Notations, conventions and standard terminology}\label{notations}

\begin{enumerate}\item The ring  $R$ is  a discrete valuation ring with quotient field $K$ and residue field $k$. An uniformizer is denoted by $\pi$, except in Section \ref{24.02.2017--1}, where we call it $t$ to avoid confusion. 
\item The spectrum of $R$ is denoted by $S$, and the quotient ring $R/(\pi^{n+1})$ is denoted by $R_n$. 
\item The characteristic of $R$ is the ordered pair $(\mathrm{char}.\,K,\mathrm{char.}\,k)$.
\item If $M$ is an $R$-submodule of $N$, we say that $M$ is \textbf{saturated} in $N$ if $N/M$ has no $\pi$-torsion. 
\item Given an object $X$ over $R$ (a scheme, a module, etc), we sometimes find useful to write $X_k$ instead of $X\ot_Rk$, $X_K$ instead of $X\ot_RK$, etc. If context prevents any misunderstanding, we also employ $X_n$ instead of $X\ot_RR_n$.  


\item To avoid repetitions, by a \textbf{group scheme} over some ring $A$, we understand an \textbf{affine group scheme} over $A$. 
\item The category of {\it flat} group schemes over a ring $A$ will be denoted by $(\mathbf{FGSch}/A)$. 
\item If $V$ is a free $R$-module of finite rank, we write $\mathbf{GL}(V)$ for the general linear group scheme representing $A\mapsto\mathrm{Aut}_A(V\ot A)$. If $V=R^n$, then $\mathbf{GL}(V)=\mathbf{GL}_{n}$.  
\item 
If $G$ is  a group scheme over $R$, we let $\mathrm{Rep}_R(G)$ stand for the category of representations of $G$ which are, as $R$-modules,  \emph{of finite type} over $R$. (We adopt Jantzen's  and Waterhouse's definition of representation. See \cite[Part I,  2.7 and 2.8, 29ff]{jantzen} and \cite[3.1-2, 21ff]{waterhouse}.)  The full subcategory of $\rep RG$ having as objects those underlying a free $R$-module is denoted by $\repp RG$. (In \cite{DHdS15} we employed the more confusing notation $\rep RG^o$.)
\item
Given an arrow $\rho:G\to H$ of group schemes over $R$, we let $\rho^\#:\rep RH\to\rep RG$ stand for the associated functor. 

\item If $\Ga$ is an abstract group and $A$ is a commutative ring, we  write $A\Ga$ for the group ring of $\Ga$ with coefficients in $A$. We let $\rep{A}{\Ga}$ stand for the category of left $A\Ga$-modules whose underlying $A$-module is of finite type. The full subcategory of $\rep{A}{\Ga}$ whose objects are moreover projective $A$-modules shall be denoted by 
$\repp{A}{\Ga}$. 

\item For an affine scheme $X$ over $R$, we call $\co(X)$ the ring of functions of $X$, and denote it by    $R[X]$.  More generally, if $A$ is any $R$-algebra, we write $A[X]$ to denote $\mathcal O(X\ot_R A)$.

\item If $G\in(\bb{FGSch}/R)$, its right-regular module is obtained by letting $G$ act on $R[G]$ via right translations. It is denoted by $R[G]_{\rm right}$ below. 
\item For an affine $R$-adic formal  scheme $\g X$, we write $R\langle\g X\rangle$ for the ring $\mathcal O(\g X)$. If $\g X$ is the $\pi$-adic completion (or completion along the closed fibre) of some affine $R$-scheme $X$, we write $R\langle X\rangle$ or $R[X]^\wedge$ to denote $R\langle\g X\rangle$.   

\item A formal group scheme over $R$ is a group object in the category of affine formal $R$-adic schemes. 

\item If $X$ is a ringed space, we let $|X|$ stand for the its underlying topological space. 

\item If $G$ is a group scheme over a ring $A$, and $\g a\subset R[G]$ is the augmentation ideal, we define the conormal module $\om(G/A)$ as being the $A$-module $\g a/\g a^2$ \cite[II.4.3.4]{DG}. Its dual $A$-module  is the Lie algebra ${\rm Lie}(G)$ of $G$ (see II.4.3.6 and II.4.4.8 of \cite{DG}). Similar conventions are in force when dealing with  formal affine group schemes.  

\item By a  $\ot$-category or tensor-category we mean a $\ot$-category  ACU in the sense of \cite[I.2.4, 38ff]{saavedra}, or a tensor category in the sense of \cite[Definition 1.1, p.105]{DM82}. Unless  mentioned otherwise, all $\ot$-functors are  ACU \cite[I.4.2]{saavedra} and all $\ot$-natural transformation are unital (which are the conventions of \cite[Section 1]{DM82}).  
\end{enumerate}
\subsection*{Acknowledgements} We thank heartily the referee for a careful reading and for many comments which improved the structure of this work.

\section{Preliminary material}\label{29.03.2017--1}
For the convenience of the reader  we briefly recall some thoughts from \cite{DHdS15} which are employed here. More detailed explanations are to be found in \cite{DHdS15} and its references.

\subsection{The standard sequences \cite[Section 2]{DHdS15}, \cite[Section 1]{WW}}\label{20.02.2017--1}

Let 
\[\rho:G'\aro G\]
be an arrow of $(\bb{FGSch}/R)$ 
 inducing an isomorphism on generic fibres. 
We then associate to $\rho$ its \emph{standard sequence}: 
\[\xymatrix{&G'\ar[d]_{\rho_{n+1}}\ar[drr]^{\rho_0=\rho}&&
\\
\cdots\ar[r] & G_{n+1}\ar[r]_-{\ph_n} &  \cdots\ar[r]_-{\ph_0}&  G_0=G}\]
as follows. 
The arrow $\ph_0:G_1\to G_0$ is the Neron blowup of $B_0:=\mathrm{Im}(\rho_0\ot k)$, and $\rho_1$ is the morphism obtained by the universal property. Now,  if $\rho_n:G'\to G_n$ is defined, then $\ph_n:G_{n+1}\to G_n$ is the Neron blowup of 
\[
B_n:=\mathrm{Im}(\rho_n\ot k)
\]
and $\rho_{n+1}$ is once more derived from the universal property.  
In particular, for each $n$, the morphisms $\rho_n\ot k:G'\ot k\to B_n$ and  $\ph_n:B_{n+1}\to B_{n}$   are faithfully flat because of \cite[Theorem 14.1]{waterhouse}.
In \cite[Theorem 2.11]{DHdS15}, it is proved that the morphism \[G'\aro\lip_n G_n\]
obtained from the above commutative diagram is an isomorphism of group schemes.

\subsection{Blowing up formal subgroup schemes \cite[Section 5]{DHdS15}}\label{23.02.2017--3} We assume that $R$ is \emph{complete}. 
Let $G$ be a group scheme over $R$ which is flat and of finite type. Denote by $\wh G$ its  $\pi$-adic completion (or completion along the closed fibre): it is a group object in the category of $R$-adic affine formal schemes. 
Let $\g H\subset\wh G$ be a closed, \emph{flat}  formal subgroup scheme and let $I_n\subset R[G]$ be the ideal of  $H_n:=\g H\ot R_n$; note that $\pi^{n+1}\in I_n$ and that $R\langle\g H\rangle=\lip_nR_n[G]/I_nR_n[G]$ (since all ideals in $R\langle G\rangle$ are closed).  We then define the Neron blowup of $\g H$, $\cn_{\g H}^\infty(G)$, as being the group scheme whose Hopf algebra, viewed as a subring of $K[G]$, is 
\[
\lid_n R[G][\pi^{-n-1}I_n].
\]
By definition, $\cn_{\g H}^\infty(G)$ lies in $(\bb{FGSch}/R)$; it usually fails to be  of finite type over $R$, and is the source of a morphism of group schemes $\cn_{\g H}^\infty(G)\to G$  inducing an isomorphism on generic fibres. One fundamental feature of $\cn_{\g H}^\infty(G)$ which is worth mentioning here is that $\cn_{\g H}^\infty(G)\ot R_n\to G\ot R_n$ induces an isomorphism between the source and $H_n$ \cite[Corollary 5.11]{DHdS15}. 

Let us profit to observe that $\cn_{\g H}^\infty(G)$ represents the  sub-functor of ${\rm Hom}_R(R[G],-)$ given by  
\[
\left\{\text{$R$-flat algebras}\right\}\aro {\rm (Groups)},\]
\[A\longmapsto\left\{g\in{\rm Hom}_R(R[G],A)\,:\,\begin{array}{c}\text{$\hat g:R\langle G\rangle\to \hat A$}\\\text{ factors through $R\langle\g H\rangle$ }\end{array}\right\}.
\]

If $H\subset G$ is a closed, flat subgroup scheme, we let $\cn_H^\infty(G)$ stand for the Neron blowup of the (necessarily $R$-flat) closed subgroup $\widehat H\subset \wh G$. It is a simple matter to verify that $R[\cn_H^\infty(G)]=\lid_n R[G][\pi^{-n-1}I]$, where $I$ is the ideal of $H$. 

\begin{ex}[See {\cite[Example 1.1]{DHdS15}}]\label{10.07.2014--1} 
Consider the pro-system of affine group schemes
\begin{equation}\label{24.09.2018--1}
\ldots\longrightarrow \GG_{a,R}\xrightarrow{\times \pi}\GG_{a,R}\xrightarrow{\times \pi}\GG_{a,R},
\end{equation}
which corresponds, on the level of rings, to the inductive system 
\begin{equation}\label{24.09.2018--2}
R[x_0]\longrightarrow R[x_1]\longrightarrow\ldots,\qquad x_i\longmapsto \pi x_{i+1}.
\end{equation}
The limit of diagram \eqref{24.09.2018--1} is simply $\cn_{\{e\}}^\infty (\GG_a)$ and the associated ring, the colimit of  \eqref{24.09.2018--2}, is   $\left\{ P\in K[x_0]\, | \,P(0)\in R\right\}$. 
Note that the $R$-module $R[G]$, being isomorphic to $R\oplus K\oplus K\oplus\cdots$, is \emph{not} projective and that $G\otimes R_n$ is always the trivial group scheme over $R_n$. 
\end{ex}

\subsection{Differential Galois groups  \cite[Section 7]{DHdS15}}\label{22.08.2020--1}

Let $X$ be a smooth $R$-scheme and   $\cd(X/R)$ be the ring of $R$--linear differential operators on $X$, see \ega{IV}{4}{16.8} or \cite[\S2]{BO}. The existence of a morphism of rings $\co_X\to\cd(X/R)$ allows us to see any left $\cd(X/R)$-module as an $\co_X$-module and we let $\modules{\cd(X/R)}$ stand for the category of left $\cd(X/R)$-modules which, when regarded as $\co_X$-modules, are coherent.  As is well-known, in case $R$ is a $\QQ$-algebra, the category $\modules{\cd(X/R)}$ is equivalent to the category of integrable connections \cite[Theorem 2.15]{BO}. 

The category $\modules{\cd(X/R)}$ possess a natural tensor product which, on the level of coherent sheaves, is simply the standard tensor product. In addition, an object $\cm\in\modules{\cd(X/R)}$
is locally free (as an $\co_X$-module) if and only if it is free of $\pi$--torsion. As in \cite[Section 7]{DHdS15}, the full subcategory of $\modules{\cd(X/R)}$ formed by those objects which are locally free as $\co_X$-modules will be  denoted by $\modules{\cd(X/R)}^\circ$.
These facts allow us to apply techniques from Tannakian categories.

We now suppose that $X$ has geometrically connected fibres over $R$ and is endowed with an $R$-point $x_0$. Let $\cm\in\modules{\cd(X/R)}^\circ$ be given; for $a,b\in\NN$, we  put $\bb T^{a,b}\cm=\cm^{\ot a}\ot \cm^{\vee\ot b}$ (here $\cm^\vee$ is the dual of $\cm$). The full subcategory of $\modules {\cd(X/R)}$ whose objects are subquotients (= quotients of subobjects) of direct sums of the form $\bigoplus_i\bb T^{a_i,b_i}\cm$ is denoted 
$\langle\cm\rangle_\ot$. The functor ``fibre at $x_0$'', 
\[x_0^*:\langle\cm\rangle_\ot\aro\modules R,\]
produces, by Tannkian duality (cf. \cite[II.4.1, 152ff]{saavedra} or \cite[Theorem 1.2.6]{DH14}), a flat group scheme $\mm{Gal}'(\cm)$ and an equivalence 
\[\ov x_0^*:
\langle\cm\rangle_\ot\ \arou\sim\rep R{\mm{Gal}'(\cm)} 
\]
of $R$-linear tensor categories. The group 
$\mm{Gal}'(\cm)$ is called the full differential Galois group of $\cm$. Contrary to the case of a field, the natural representation 
\[
\rho:\mm{Gal}'(\cm)\aro \bb{GL}(x_0^*\cm)
\]
obtained from   $\ov x^*_0$ is not necessarily a closed immersion. On the other hand it is possible to factor $\rho
$ as 
\[
\mm{Gal}'(\cm)\arou{\si}  \mm{Gal}(\cm)\arou\tau \bb{GL}(x_0^*\cm),
\]
where $ \mm{Gal}(\cm)$ is a flat group scheme called the restricted differential Galois group. The arrow $\si$ induces an isomorphism of generic fibres  and 
$\tau$ is a closed immersion.

\section{The ubiquity of blowups of formal subgroups}\label{06.02.2017--1} In this section, we assume that $R$ is \emph{complete}. Let $G'\in(\bb{FGSch}/R)$ have a generic fibre $G'\ot K$ of finite type (over $K$). This being so, there exists some $G\in(\bb{FGSch}/R)$, now of finite type over $R$, which is the target of a morphism 
\[\rho:G'\aro G\]
inducing an isomorphism on generic fibres.  Then, as explained on Section \ref{20.02.2017--1}, we have the standard sequence of $\rho$:
\[\xymatrix{&G'\ar[d]_{\rho_{n+1}}\ar[drr]^{\rho_0=\rho}&&
\\
\cdots\ar[r] & G_{n+1}\ar[r]_-{\ph_n} &  \cdots\ar[r]_-{\ph_0}&  G_0=G,}\]
where each $G_n$ is of finite type over $R$ (by construction). 
To state Theorem \ref{20.02.2017--2} below, we let  
\[\begin{split}
B_n&= \text{image of  $\rho_n\ot k:G'\ot k\to G_n\ot k$}
\\&=\text{centre of $\ph_n:G_{n+1}\to G_n$}.
\end{split}\]

\begin{thm}\label{20.02.2017--2}\begin{enumerate}\item Assume that for all  $n$, the morphisms $\ph_n:B_{n+1}\to B_n$ are isomorphisms. Then
\begin{enumerate}\item  the morphism induced between $\pi$-adic completions  
$R\langle G\rangle \to R\langle G'\rangle$
is surjective.  
\item Let $\g H$ stand for the closed formal subgroup scheme of $\wh G$ cut out by the kernel of $R\langle G\rangle \to R\langle G'\rangle$. (Clearly $\wh G'\simeq\g H$.) Then there exists an isomorphism of group schemes 
\[\si:G'\arou\sim \cn_{\g H}^\infty(G)\]
rendering  
\[\xymatrix{
G'\ar[r]^-{\si}\ar[dr]_\rho & \cn_{\g H}^\infty(G)\ar[d] \\ & G  }\]
commutative.  
\end{enumerate}
\item Assume that the characteristic of $k$ is zero. Then, there exists $n_0\in\NN$ such that, for all $n\ge n_0$, the arrows $\ph_{n}:B_{n+1}\to B_n$
are isomorphisms. 
\end{enumerate}
\end{thm}

The proof shall need  
\begin{lem}\label{20.02.2017--3}The obvious morphisms $G'\ot k\to B_n$ induce an isomorphism 
\[
G'\ot k\arou{\sim}\lip_nB_n.
\]  
\end{lem}

\begin{proof}
One considers the commutative diagram of $k$-algebras obtained from the construction of the standard sequence
\[
\xymatrix{  \cdots\ar[r]  &\ar@{->>}[d]k[G_{n-1}]\ar[r]& k[G_n]\ar[r]\ar@{->>}[d]\ar@{->>}[d] & k[G_{n+1}]\ar@{->>}[d] \ar[r]&\cdots \\
\ar[r] \cdots &k[B_{n-1}]\ar[ru]\ar@{^{(}->}[r]  & k[B_n]\ar@{^{(}->}[r] \ar[ru] & k[B_{n+1}]\ar[r] &\cdots   }   
\]
(As usual, $\xymatrix{\ar@{^{(}->}[r]&}$ means injection, and $\xymatrix{\ar@{->>}[r]&}$ surjection.)
It is then clear that 
\[\lid_nk[G_n]\aro \lid_nk[B_n]\]
is an isomorphism. As explained in Section \ref{20.02.2017--1},  the canonical arrow $G'\to\lip G_n$ induces an isomorphism. Since taking tensor products commutes with direct limits \cite[Theorem A.1]{matsumura}, the result follows.    
\end{proof}

\begin{proof}[Proof of Theorem \ref{20.02.2017--2}](1a) The  morphism $G'\ot k\to G\ot k$ induces an isomorphism between $G'\ot k$ and $B_0$ because of the assumption and  of Lemma \ref{20.02.2017--3}. Hence, the arrow between $k$-algebras 
$k[G]\to k[G']$ is surjective. As is well-known (see for example \cite[10.23(ii), p. 112]{AM}),  the induced arrow between completions $R\langle G\rangle\to R\langle G'\rangle$ is then surjective.  

(1b) Let $I_n\subset R[G]$ be the ideal of the closed immersion $G'\ot R_n\to G\ot R_n\to G$ (in particular $\pi^{n+1}\in I_n$). Writing $\cn$ instead of $\cn_{\g H}^\infty(G)$, the definition says that  \[R[\cn]=\bigcup_nR[G][\pi^{-n-1}I_n].\] Since, by construction, the   morphism of $R$-algebras $R[G]\to R[G']$ sends $I_n$ to the ideal $\pi^{n+1}R[G']$, we conclude that $R[G]\to R[G']$ factors thought $R[\cn]$, that is,   $\rho:G'\to G$ factors through $\cn\to G$.  
Using that $\cn \ot k\to G\ot k$ 
induces an isomorphism between $\cn \ot k$ and $\g H\ot k$ \cite[Corollary 5.11]{DHdS15}, we conclude that $G'\ot k\to\cn\ot k$ is also an isomorphism. Together with the fact that $G'\ot K\to\cn\ot K$ is an isomorphism, this easily implies that $G'\to\cn$ is an isomorphism \cite[Lemma 1.3]{WW}.

(2) As   already remarked in Section \ref{20.02.2017--1}, the morphisms 
\[\ph_n|_{B_{n+1}}:B_{n+1}\aro B_n\]
are all faithfully flat. 
It then follows from \cite[${\rm VI}_B$, Proposition 1.2, p.335]{SGA3} that 
\[\tag{$*$}
\dim B_{n+1}= \dim B_n+\dim \mathrm{Ker}(\ph_n|_{B_{n+1}}).
\] 
Using the equality 
$\dim G_n\ot k=\dim G_n\ot K$  \cite[$\mathrm{VI}_B$, Corollary 4.3, p.358]{SGA3}, we conclude that 
\[\begin{split}
\dim B_n&\le \dim G_n\ot k
\\
&=\dim G_n\ot K
\\
&=\dim G\ot K.
\end{split}\]
Let $n_0\in\NN$ be such that $\dim B_{n_0}$ is maximal. We then derive from equation ($*$) above that for all $n\ge n_0$, 
\[\tag{$**$}
\dim \mathrm{Ker}(\ph_n|_{B_{n+1}})=0. 
\]

Because of the assumption on the characteristic, 
we know that $B_n$ and $G_n\ot k$ are \emph{smooth}     $k$-schemes, so that $B_n\ot k\to G_n\ot k$ is a regular immersion [EGA $\mathrm{IV}_4$, 17.12.1, p.85]. Now, as explained in \cite[Section 2.2]{DHdS15}, the group scheme   
$\mathrm{Ker}\,(\ph_n\ot k)$
is isomorphic to a vector group $\GG_{a,k}^{c_n}$ for some integer $c_n\ge0$. Consequently, $\mathrm{Ker}(\ph_n|_{B_{n+1}})$ is isomorphic to a closed subgroup scheme of $\GG_{a,k}^{c_n}$; the hypothesis on the characteristic then implies that  $\mathrm{Ker}(\ph_n|_{B_{n+1}})$ is either trivial or positive dimensional (by \cite[Proposition IV.2.4.1]{DG}, say). 
Equation ($**$) then shows that $\mathrm{Ker}(\ph_n|_{B_{n+1}})=\{e\}$ if $n\ge n_0$. We have therefore proved that $\ph_n|_{B_{n+1}}$ is a closed immersion \cite[${\rm VI}_B$, Corollary 1.4.2, p.341]{SGA3}, and consequently an isomorphism if $n\ge n_0$.
\end{proof}

\begin{cor}\label{11.07.2017--2} If $k$ is of characteristic zero, there exists $n_0\in\NN$ and a formal  closed flat subgroup $\g H\subset\wh G_{n_0}$ such that  $G'\simeq\cn_{\g H}^\infty(G_{n_0})$.\qed
\end{cor}

\begin{rmk}Statement (2) in Theorem \ref{20.02.2017--2} fails if $k$ is allowed to have characteristic $p>0$ as the following construction shows. Let $k$ and $K$ be of  characteristic $p$ (the case of mixed characteristic will be considered elsewhere).  Let  $A_0= R[x_0]$ be a polynomial algebra in one variable. As in Example \ref{10.07.2014--1}, we introduce $\tilde A=\{P\in K[x_0]\,:\,P(0)\in R\}$ and define 
\[
\De:\tilde A\aro\tilde A\ot_R\tilde A,\quad \ep:\tilde A\to R ,\quad\text{and}\quad \si:\tilde A\aro\tilde A
\]
by requiring that $\De x_0=x_0\ot1+1\ot x_0$, $\ep x_0=0$ and  $\si(x_0)=-x_0$. These morphisms give $A_0$ and $\tilde A$ the structure of a commutative Hopf algebra over $R$. (Clearly the group scheme ${\rm Spec}\,A_0$ is just $\GG_{a,R}$.) 
We construct inductively elements $x_n$ of $\tilde A$ by putting $x_{n+1}=\pi^{-1}(x_n^p-x_n)$ and then write $A_n:=R[x_0,\ldots,x_n]$, which is an $R$-subalgebra of $\tilde A$. A simple argument shows that
\[
\De x_n=x_n\ot1+1\ot x_n,\quad\ep(x_n)=0,\quad\text{and}\quad \si(x_n)=-x_n. 
\]
In particular, $\De(A_n)\subset A_n\ot A_n$ and $\si(A_n)\subset A_n$. 
Let $G_n$ be ${\rm Spec}\, A_n$ and give it the structure of a group scheme arising from $\De$, $\ep$ and $\si$. It follows that $\{x_n^p-x_n=0\}$ cuts out a closed subgroup scheme $B_n$ of $G_{n}\ot k$, and that the natural arrow $G_{n+1}\to G_n$ is simply the blowup of $B_n$. Proceeding by induction, one shows that $A_{n+1}=A_n[T]/(\pi  T-(x_n^p-x_n))$ and that the natural arrow $k[B_n]\to k[G_{n+1}]$ identifies $k[G_{n+1}]$ with a polynomial algebra in one variable over $k[B_n]$. (Here we apply the description of the blowup of an ideal generated by a regular sequence given in [BA-X, Theorem 1, \S9, no.7].) Consequently,  $B_{n+1}\to B_n$ is faithfully flat and has a non-trivial kernel, which is isomorphic to $\ZZ/p\ZZ$.
\end{rmk}

We now wish to derive from Theorem \ref{20.02.2017--2} some consequences concerning unipotent group schemes over $R$ of characteristic $(0,0)$; all hinges on the well-known principle that, in this case, Lie subalgebras are ``algebraic.'' 
Recall that $U\in(\bb{FGSch}/R)$ is called \emph{unipotent} if $U\ot K$, respectively $U\ot k$, is a unipotent group scheme over $K$, respectively $k$. 
(It should be noted that in certain cases it is sufficient to require this condition only for $K$. See for example Theorem 2.11 of \cite{tong}.) 

Let us prepare the terrain: Assume  that $R$ has characteristic $(0,0)$ and  that  $U\in(\bb{FGSch}/R)$  is unipotent and of finite type; we also abbreviate $\om_U=\om(U/R)$. A most fundamental tool for studying $U$ is the exponential morphism 
\[
\exp:{\rm Lie}(U)_a\aro U
\]
which is explained in terms of functors in \cite[II.6.3.1, 264ff]{DG}. (In loc.cit. the authors work over a base field, but this assumption is unnecessary for the construction.) Endowing ${\rm Lie}(U)_a$ with its BCH multiplication, $\exp$ becomes an \emph{isomorphism} of group schemes \cite[Section 1.3]{tong}. 
Since we wish to deal with formal group schemes as well, it is convenient to elaborate on a purely algebraic expression for $\exp^\#:R[U]\to\bb S(\om_U)$. (Recall that ${\rm Lie}(U)_a={\rm Spec}\,\bb S(\om_U)$.)

Write $D:R[U]\to\om_U$ for the canonical $\ep$-derivation; this gives rise to an $\ep$-derivation, denoted likewise, $D:R[U]\to\bb S(\om_U)$ into the symmetric $R$-algebra. For each $n>0$, we write $D^n:R[U]\to\bb S(\om_U)$ for the convolution product of $D$ with itself $n$ times. In terms of Sweedler's notation, this reads 
\begin{equation}\label{convolution}
D^n(f)=\sum_{(f)}D(f_{(1)})\cdots D(f_{(n)}),\qquad f\in R[U].
\end{equation}
Then, agreeing to fix $D^0(f)=\ep(f)\po1$, we define 
\[
\eta_t(f):=\sum_{n=0}^\infty\fr{D^n}{n!}(f)t^n,\quad f\in R[U].  
\](This is an element of $\pos{\bb S(\om)}t$.)
Leibniz' formula $D^n(\ph\ps)=\sum_{m=0}^n\binom{n}{m}D^m(\ph)D^{n-m}(\ps)$ proves that $\eta_t$ is a morphism of $R$-algebras, while the expression 
\[
D^{\ell+m}(f)=\sum_{(f)}D^\ell(f_{(1)})D^m(f_{(2)})
\]
allows one to verify that $\eta_t$ satisfies the uniqueness statement concerning the exponential given in \cite[II.3.6.1, p. 264]{DG}.  
In conclusion, \[\exp^\#=\sum_{n=0}^\infty\fr{D^n}{n!}.\]
We are now in a position to comfortably prove the

\begin{prp}\label{11.07.2017--1}We keep the notations and assumptions made just above.  Write $\g U$ for the $\pi$-adic completion of $U$ and let $\g H\subset\g U$ be a formal flat closed subgroup scheme. Then,  there exists a closed flat subgroup $H\subset U$ inducing $\g H$.  Said differently, these formal subgroup schemes are algebraizable. 
\end{prp}

\begin{proof} We start by remarks on the structure of $\g H$. 
We note that $\g H$ is connected because the unipotent group scheme $\g H\ot k$ is,  as a scheme, an affine space  \cite[Proposition IV.2.4.1, p.497]{DG}. Also, since $\g H\ot k$ is smooth over $k$, the fibre-by-fibre smoothness criterion \cite[${\rm IV}_4$, 17.8.2, p.79]{EGA} proves that  $\g H\ot R_n$ is smooth over $R_n$. We conclude that the conormal module $\om(\g H\ot R_n/R_n)$ is free over $R_n$; as the canonical arrow $\om(\g H/R)\ot R_n\to\om(\g H\ot R_n/R_n)$ is bijective \cite[I.4.1.6, p.99]{DG}, $\om(\g H/R)$ is free over $R$.

 The canonical morphism $R[U]\to R\langle U\rangle$ induces a bijection $\om(U/R)\to\om(\g U/R)$ so that we can associate to $\om(\g H/R)$, which is a quotient Lie-coalgebra  of $\om(\g U/R)$, a quotient Lie coalgebra $\om(\g H/R)$ of $\om(U/R)$. Let $H\subset U$ be a closed subgroup scheme which  corresponds, via $\exp$, to $\om(\g H/R)$. (In this case $H\simeq\mm{Spec}\,\bb S( \om(\g H/R))$ and the kernel of $\om(U/R)\to\om(H/R)$ is the kernel of $\om(U/R)\to\om(\g H/R)$.)

\begin{claim*}The closed immersion $\wh H\to\g U$ factors as $\wh H\to \g H\to \g U$, where $\wh H\to \g H$ is a closed immersion inducing an isomorphism $\om(\g H/R)\stackrel{\sim}\to\om(\wh H/R)$.  
\end{claim*}
We ease notation by putting $\om_{\g H}=\om(\g H/R)$, etc. Write $\nu:=\mm{Ker}(\om_U\to\om_{\g H})$ and let $I\subset R\langle U\rangle$ stand for the ideal cutting out  $\g H$. Let $\wh D^n$, respectively  $\wh\exp^\#$, stand for the prolongation of $D^n$, respectively $\exp^\#$, to $\pi$-adic completions. Then $\wh D(I)\subset \nu$, and using equation \eqref{convolution}, we conclude that $\wh\exp^\#$ takes $I$ to the ideal $(\nu)\po\bb S(\om_U)^\wedge$, which means that $\g H\to \g U$ factors $\wh H\to\g U$. Since $\wh H\to\g U$ is a closed immersion so is $\wh H\to\g H$; since $\om_{\g U}\to\om_{\wh H}$ and $\om_{\g U}\to\om_{\g H}$ have the same kernel, the claim is proved. 

 Using again that $\om(\g H/R)\ot R_n\stackrel\sim\to\om(\g H_n/R_n)$ and $\om(H/R)\ot R_n\stackrel\sim\to\om(H_n/R_n)$ \cite[I.4.1.6, p.99]{DG}, we conclude that the closed immersion $H_n\to \g H_n$ induces an isomorphism $\om(\g H_n/R_n)\to\om(H_n/R_n)$. Now, \cite[I.4.4.2, p.109]{DG}   shows that $\wh\co_{H_n,e}$, resp. $\wh\co_{\g H_n,e}$,  is isomorphic to the completion of   ${\bf S}(\om(H_n/R_n))$, resp. ${\bf S}(\om(\g H_n/R_n))$, with respect to the augmentation ideal, so that we derive an isomorphism $\wh\co_{\g H_n,e}\stackrel{\sim}{\to}  \wh\co_{H_n,e}$. This,  in turn implies $\co_{\g H_n,e}\stackrel{\sim}{\to} \co_{H_n,e}$. Hence,   $H_n\to\g H_n$ is flat at the point $e$, and consequently flat allover \cite[${\rm VI}_B$, Proposition 1.3]{SGA3}. We can therefore say that  the closed immersion $H_n\to \g H_n$ is an open immersion. Because  
$H_n$ and $\g H_n$ are connected, we conclude that $H_n\to \g H_n$ is an isomorphism which,  in turn, implies that $\wh H\to\g H$ is an isomorphism. 
\end{proof}

The following corollary shall prove useful in Section \ref{07.07.2017--1}.  
\begin{cor}\label{29.09.2017--1}Assume that $R$ is of characteristic $(0,0)$ and let  $U\in(\bb{FGSch}/R)$ be unipotent and of finite type. 
Given an arrow 
\[\rho:G\aro U\]
of $(\bb{FGSch}/R)$ which induces an  isomorphism on generic fibres, either $G$ is of finite type, or $R[G]$ contains an $R$-submodule isomorphic to $K$.    
\end{cor}

\begin{proof}We assume that $G$ is not of finite type over $R$. Let 
\[
\cdots\aro U_{n}\aro \cdots\aro U
\]
be the standard sequence of $\rho$. It follows from Corollary  \ref{11.07.2017--2} that for some $n_0\in\NN$, $G\simeq \cn^\infty_{\g H}(U_{n_0})$, where $\g H$ is a closed and flat subgroup of $\wh U_{n_0}$. Since the kernel of $U_{n+1}\ot k\to U_n\ot k$ is unipotent \cite[Theorem 1.5]{WW}, basic theory \cite[IV.2.2.3, p.485]{DG} tells us that  $U_{n_0}$ is unipotent. So  Proposition \ref{11.07.2017--1} can be applied: there exists some flat and closed  $H\subset U_{n_0}$ such that   
 $\g H=\wh H$. In addition, as $G$ is not of finite type over $R$, the closed immersion $H\subset U_{n_0}$ is not an isomorphism.  
Now, if $f\in R[U_{n_0}]\setminus\{0\}$ belongs to the ideal cutting out $H$, then, in  $R[G]\simeq\cn_{H}^\infty(U_{n_0})$, $f$ can be \emph{uniquely} divided by any power of $\pi$, so that $\cup_mR\pi^{-m}f$ is isomorphic to $K$.   
\end{proof}

\section{Neron blowups of formal subgroup schemes in the Tannakian theory of ``abstract'' groups}\label{abstract_groups}

In Section \ref{06.02.2017--1} we saw that blowing up a formal subgroup scheme is a very pertinent operation in the theory of flat group schemes over $R$. We now investigate if such group schemes play a role in the Tannakian theory of abstract groups. Apart from the results of Section \ref{24.02.2017--1}, our motivation to carry the present study comes from the fact that Tannakian categories of abstract groups produce easily many interesting examples of group schemes. 

Section \ref{10.02.2017--1} contains preparatory material on the Tannakian categories in question. Section \ref{06.07.2017--3} explains how to mimic, in our setting, the folkloric computation of differential Galois groups from the monodromy representation. 
Section \ref{23.02.2017--7} shows that Neron blowups of formal subgroup schemes appear naturally in the Tannakian theory of abstract groups, already in characteristic $(0,0)$. The results here shall be employed in Section \ref{24.02.2017--1} to study categories of $\cd$-modules.

\subsection{Representations of abstract groups}\label{10.02.2017--1}\label{06.07.2017--4} (In this section we make no supplementary assumption on $R$.)

Let $\Ga$ be an abstract group. In what follows, we convey some thoughts on the existence and basic properties of a ``Tannakian envelope'' for $\Ga$ over $R$, that is, a \emph{flat} affine group scheme $\Pi$ such that $\rep{R}{\Pi}\simeq \rep{R}{\Ga}$.  Of course, contrary to the case of a ground field, flatness is not gratuitous, so that the existence of such an envelope imposes one extra property on $\rep R\Ga$  which we briefly explain. Recall from \cite[Definition 1.2.5]{DH14} that $\rep{R}{\Ga}$ is \emph{Tannakian} if every  $V\in\rep{R}{\Ga}$ is the target of an epimorphism $\wt V\to V$ from an object $\wt V\in\repp{R}{\Ga}$. If $\rep R{\Ga}$ is not Tannakian, it still contains the full Tannakian subcategory $\rep R{\Ga}^\text{tan}$ consisting of those objects satisfying the aforementioned condition.

\begin{dfn}\label{22.02.2017--3}
\begin{enumerate}\item 
We say that $\Ga$ is \emph{Tannakian over $R$} if       $\rep{R}{\Ga}$ is Tannakian in the sense explained above.   
\item  The \emph{Tannakian envelope} of $\Ga$ over $R$  is the flat   group scheme  constructed from $\rep{R}{\Ga}^{\rm tan}$ and the forgetful functor $\rep{R}{\Ga}^{\rm tan} \to\modules R$ by means of the main theorem of Tannakian duality (see \cite[II.4.1, 152ff]{saavedra} or \cite[Theorem 1.2.6]{DH14}). Note that, whether $\Ga$ is Tannakian or not, its  envelope does exist. 

\end{enumerate}
\end{dfn}

\begin{ex}
Any \emph{finite} abstract group $\Ga$ is Tannakian. (One uses that the multiplication morphism  $R\Ga\ot_RV\to V$ is equivariant.)   
\end{ex}

In order to explore the defining property of a Tannakian group,   we use the following terminology. 
 
\begin{dfn}Let $V\in\mm{Rep}_R(\Ga)$. 
\begin{enumerate}\item Assume that $\pi V=0$, that is, $V$ is a $k$-module. An object  $\wt V\in\repp{R}{\Ga}$ such that $\wt V/\pi\wt V\simeq V$ is called a \emph{lift} of $V$ from $k$ to $R$. 
\item An objet $\wt V\in\repp{R}{\Ga}$ which is the source of an epimorphism $\wt V\to V$ is called a \emph{weak lift} of $V$. 
\end{enumerate}
\end{dfn}

\begin{prp}[{cf. \cite[Proposition 5.2.2]{DH14}}]\label{21.02.2017--1} Assume that each $V\in\mm{Rep}_R(\Ga)$ which is annihilated by $\pi$ has a weak lift. Then every $E\in\mm{Rep}_R(\Ga)$ has a weak lift.
\end{prp}

\begin{proof}Given $M\in\rep R\Ga$, we define  
\[r(M)=\min\{s\in\NN\,:\,\pi^sM_{\rm tors}=0\}.
\]
We shall proceed by induction on $r(E)$, the case $r(E)=0$ being trivial. 
Assume $r(E)=1$, so that $\pi E$ is torsion-free. Let $q:E\to C$ be the quotient by $\pi E$; since $C$ is annihilated by $\pi$, the hypothesis gives us $\wt C\in\repp{R}{\Ga}$ and an epimorphism  $\si:\wt C\to C$.  
We  then have a commutative diagram with  exact rows 
\[
\xymatrix{0\ar[r]& \pi E\ar[r]&E\ar[r]^q&C\ar[r]&0\\ 0\ar[r]& \pi E\ar[u]^\sim\ar[r]&\wt E\ar@{}[ru]|\square\ar[u]^{\tau}\ar[r]_\chi&\ar[u]_\si\wt C\ar[r]&0,}
\]
where the rightmost square is cartesian and $\tau$ is surjective. Since $\pi E$ and $\wt C$ are torsion-free, so is $\wt E$, and we have found a weak lift of $E$.  

Let us now assume that $r(E)>1$.  Let $N=\{e\in E\,:\,\pi e=0\}$ and denote by $q:E\to C$ the quotient by $N$. It then follows that $\pi^{r(E)-1}C_\mm{tors}=0$, so that $r(C)\le r(E)-1$.  
By induction there exists  $\wt C\in\repp R\Ga$ and a surjection $\si:\wt C\to C$. We arrive at commutative diagram with exact rows 
\[
\xymatrix{0\ar[r]& N\ar[r]&E\ar[r]^q&C\ar[r]&0\\ 0\ar[r]& N\ar[u]^\sim\ar[r]&\wt E\ar@{}[ru]|\square\ar[u]^{\tau}\ar[r]_\chi&\ar[u]_\si\wt C\ar[r]&0,}
\]
where the rightmost square is cartesian and $\tau$ is surjective.  Since $\wt C$ is torsion-free as an $R$-module, we conclude that $\wt E_\mm{tors}= N$, so that $r(\wt E)\le1$. We can therefore find  $\wt E_1$ and a surjection  $\wt E_1\to\wt E$ and consequently  a surjection $\wt E_1\to E$. \end{proof}

The simplest cases where any $V\in\rep R\Ga$ annihilated by $\pi$ has a lift  occurs when $R$ has a coefficient field \cite[p.215]{matsumura} (in this case the reduction morphism $\GL_r(R)\to\GL_r(k)$ has a section) or when the group in question is free. 
Hence: 

\begin{cor}\label{23.02.2017--2}A free group is Tannakian. If  $R$ has a coefficient field, then any abstract group is Tannakian.\qed
\end{cor}

\begin{cor}\label{12.06.2017} Let $\Pi$  denote the Tannakian envelope of $\Ga$ over $R$. Let $\Te$ denote the Tannakian envelope of $\Ga$ over $k$. 
The following claims hold true. 
\begin{enumerate}\item There exists a canonical  faithfully flat arrow of group schemes 
$h:\Te\to \Pi\ot k$.
\item The   morphism $h$ is an isomorphism if and only if $\Ga$ is Tannakian. 
\end{enumerate} \end{cor}
\begin{proof}(1)  According to \cite[Part I, 10.1]{jantzen}, $\rep k{\Pi \otimes k}$ can be identified with the full subcategory of $\rep R\Pi$ consisting  of those representations annihilated by $\pi$. Hence, $\rep k{\Pi\ot k}$ can be identified with the full tensor subcategory of $\rep R\Ga^{\rm tan}$ of objects which are annihilated by $\pi$. 
We then derive a fully faithful $\ot$-functor  $\eta:\rep k{\Pi\ot k}\to\rep k\Ga$ which, on the level of vector spaces is just the identity. (Note that if $V\in\rep k{\Pi\ot k}$, then $\eta(V)$  always admits a weak lift.)
Moreover, $\eta$ is also closed under taking subobjects. Tannakian duality then produces a morphism $h:\Te\to  \Pi\otimes k$ which is, in addition, faithfully flat \cite[Proposition 2.21, p.139]{DM82}.

(2) Assume now that $h$ is an isomorphism, so that $\eta$ is an equivalence. Then,  any $V\in\rep k\Ga$ admits a weak lift, so that, by Proposition \ref{21.02.2017--1}, any $E\in\rep R\Ga$ admits a weak lift. This means that $\rep R\Ga$ is Tannakian. The converse is also simple.
\end{proof}

\begin{ex}[Non-Tannakian groups]\label{06.07.2017--2}
We assume that $R$ is of mixed characteristic $(0,p)$.  
Let $\Ga$ be a periodic group of finite exponent (all orders are divisors of a fixed integer). A theorem of Burnside   \cite[Theorem 2.9, p.40]{dixon} says that every morphims $\Ga\to\GL_r(R)$ must then have a \emph{finite image}, so that a morphism $\Ga\to \GL_r(k)$ which admits a lift must have a finite image. 
This fact puts heavy restrictions on a Tannakian periodic group of finite exponent. 
For, if  $\Ga$ is moreover Tannakian, then all representations $\rho:\Ga\to\GL_r(k)$ must have a finite image. (By the Tannakian property some $\si:\Ga\to\GL_{r+h}(k)$ of the form
\[\begin{bmatrix}*&*\\ *&\rho\end{bmatrix}
\]
admits a lift.) Hence, such a periodic group cannot be ``linear over $k$''. A specific counter-example is then given by the additive group $(k,+)$, if $k$ is infinite. 
\end{ex}

\subsection{Computing faithfully flat quotients of the Tannakian envelope}\label{06.07.2017--3}(In this section we make no supplementary assumption on $R$.)

We fix  $G\in(\bb{FGSch}/R)$ and  an abstract group  $\Ga$, whose  Tannakian envelope over $R$ is denoted by $\Pi$ (see Definition \ref{22.02.2017--3}).   In this section, we shall first explain how to associate to each arrow of abstract groups $\ph:\Ga\to G(R)$ a morphism $u_\ph:\Pi\to G$. Then, we set out to  determine under which conditions on $\ph$, the map $u_\ph$ is faithfully flat.  We inform the reader that the analogous situation over a base field is folkloric (meaning that some cursory discussions can be found in the literature, e.g.   \cite[10.24]{deligne87} and  \cite[Appendix]{ABC}).


Before starting, let us fix some notations. We write  $\om_\Ga:\rep R\Ga^{\rm tan}\to\modules R$ and $\om_G:\rep RG\to\modules R$ for the  forgetful functors.  
The set $\bb{Fun}^\ot_*(\rep RG,\rep R\Ga^{\rm tan})$ is formed by  all $\ot$-functors $\eta$  satisfying, as $\ot$-functors, the equality $\om_G\circ\eta=\om_\Ga$. Given $H\in(\bb{FGSch}/R)$,  we define   $\bb{Fun}^\ot_*(\rep RG,\rep RH)$  analogously. Elements in these sets shall be called \emph{pointed} $\ot$-functors.  Finally, if $\ga\in\Ga$ and $M\in\rep R\Ga$, respectively $g\in G(R)$ and $N\in \rep RG$, we let $\ga_M:M\to M$, respectively $g_N$, stand for the action of $\ga$ on $M$, respectively $g$ on $N$.  

Let $\ph:\Ga\to G(R)$ be a morphism of abstract groups. For each  $M\in \rep RG$ and each $\ga\in\Ga$, we have an $R$-linear automorphism $\ph(\ga)_M$. 
In this way, we obtain a pointed $\ot$-functor 
\[
\ph^\natural\colon\rep RG\aro\rep R\Ga^{\rm tan}
\]
verifying the equations 
\[
\text{$\ga_{\ph^\natural(M)}=\ph(\ga)_M$ as elements of $\mm{Aut}_R(M)$}, 
\]
for all $\ga\in \Ga$ and $M\in\rep RG$.

Conversely, let $T:\rep RG\to\rep R\Ga^{\rm tan}$ be a pointed $\ot$-functor. If $\ga\in \Ga$ is fixed, and $M\in\rep RG$ is given, then  $\ga_{TM}$ is an $R$-linear automorphism of $M$. It is a simple matter to show that the family $\{\ga_{TM}\,:\,M\in\rep RG\}$ defines a $\ot$-automorphism of $\om_G$, that is, an element of $\mm{Aut}^\ot(\om_G)$. (It is equally simple to note that $\ga\mapsto\{\ga_{TM}\,:\,M\in\rep RG\}$ is in fact an arrow of groups.) Since the canonical arrow $G(R)\to\mm{Aut}^\ot(\om_G)$ is bijective \cite[II.2.5.4, p.133]{saavedra},  we arrive at a morphism of \emph{groups} 
\[T^\flat:\Ga\aro G(R)\] 
verifying the equations 
\[
\text{$[T^\flat(\ga)]_M=\ga_{TM}$ as elements of $\mm{Aut}_R(M)$,}
\]
for all $\ga\in\Ga$ and $M\in\rep RG$.

\begin{prp}\label{26.09.2017--1}The arrows  $\ph\mapsto\ph^\natural$ and $T\mapsto T^\flat$ are inverse bijections between $\hh{}{\Ga}{G(R)}$ and $\bb{Fun}_*^\ot(\rep RG,\rep R\Ga^{\rm tan})$. In addition, the construction $(-)^\natural$ is functorial in the sense that, for each $\rho:G\to H$ in $(\bb{FGSch}/R)$, 
the diagram 
\[
\xymatrix{
\hh{}{\Ga}{G(R)}\ar[rr]^-{(-)^\natural}\ar[d]_{\rho(R)\circ(-)}&&\ar[d]^{(-)\circ\rho^\#}\bb{Fun}^\ot_*(\rep RG,\rep R\Ga^{\rm tan})
\\
\hh{}{\Ga}{H(R)}\ar[rr]_-{(-)^\natural}&& \bb{Fun}^\ot_*(\rep RH,\rep R\Ga^{\rm tan})}
\]
commutes. 
\end{prp}

\begin{proof}One immediately sees that  $[(\ph^\natural)^\flat(\ga)]_M=\ph(\ga)_M$ for any $M\in\rep RG$. Since the natural arrow $G(R)\to\mm{Aut}^\ot(\om_G)$ is bijective, we conclude that 
 $(\ph^\natural)^\flat(\ga)=\ph(\ga)$. 
 To prove that $(T^\flat)^\natural:\rep RG\to\rep R\Ga^{\rm tan}$ coincides with $T$, we need to show that for each $\ga\in\Ga$ and each $M\in\rep RG$,  the $R$-linear automorphisms    $\ga_{(T^\flat)^\natural(M)}$ and $\ga_{TM}$ are one and the same. Now,   $\ga_{(T^\flat)^\natural(M)}$ is just  $[T^\flat(\ga)]_M$, which is the only element of $G(R)$ inducing the automorphism $\ga_{TM}$. The verification of the last property is also very simple and we omit it. 
\end{proof}

Given $\ga\in\Ga$, let $u(\ga)\in\Pi(R)={\rm Aut}^\ot(\om_\Ga)$ be defined by 
\[
\text{$u(\ga)_M=\ga_M$ as elements of $\mm{Aut}_R(M)$}. 
\]
In this way, we construct a morphism of groups \[u:\Ga\aro \Pi(R)\] and then a pointed $\ot$-functor $u^\natural:\rep R\Pi\to\rep R\Ga^{\rm tan}$. Composing the canonical equivalence $\rep R\Ga^{\rm tan}\to\rep R\Pi$ with $u^\natural$, we obtain precisely the identity, so that $u^\natural$ is also a $\ot$-equivalence. Using  Proposition \ref{26.09.2017--1} and the bijection \[\hh{}{G}{H}\aro\bb{Fun}^\ot_*(\rep RH,\rep RG),\quad\rho\longmapsto\rho^\#,\] explained in \cite[II.3.3.1, p.148]{saavedra}, we arrive     at:

\begin{cor}\label{22.02.2017--4}The function 
\[
\hh{}{\Pi}{G}\aro \hh{}{\Ga}{G(R)},\quad \rho\longmapsto \rho(R)\circ u
\]
is a bijection. \qed
\end{cor}

\begin{rmk}Corollary  \ref{22.02.2017--4} gives a different way to think about the Tannakian envelope; needless to say this is a common viewpoint in algebraic Topology, see \cite[Appendix A]{ABC}. 
\end{rmk}

Using Corollary \ref{22.02.2017--4}, we can easily detect when a morphism $\Ga\to G(R)$ presents  $G$  as a faithfully flat quotient of $\Pi$.

\begin{dfn} Let   $G$ be a flat group scheme over $R$, and $\ph:\Ga\to G(R)$ a homomorphism of abstract groups. 
Write $\ph(\ga)_k:\mathrm{Spec}\,k\to G_k$, respectively $\ph(\ga)_K:\mathrm{Spec}\,K\to G_K$, for the closed immersion associated to $\ph(\ga):\mathrm{Spec}\,R\to G$ by base-change. We say that  $\ph$ is \emph{dense} if  $\{\ph(\ga)_k\}_{\ga\in\Ga}$, respectively  $\{\ph(\ga)_K\}_{\ga\in\Ga}$, is schematically dense  in $G_k$ \cite[${\rm IV}_3$, 11.10.2, p.171]{EGA}, respectively in $G_K$.  
\end{dfn}

\begin{rmk}The somewhat complicated definition given above is due to the usual discrepancy between topological and schematic properties. Needless to say, if in the above definition $G_k$ and $G_K$ are reduced, then the condition simply means that the image of $\Ga$ in $|G_k|$ and in $|G_K|$ is dense \cite[${\rm IV}_3$, 11.10.4]{EGA}. But in the case of  non-reduced schemes one has to proceed with caution as a simple example shows (take $\Ga=\{\pm1\}$, $R=\ZZ_2$ and $G=\mu_{2,R}$). Also note that  requiring  the family $\{\ph(\ga)\}_{\ga\in\Ga}$ to be schematically dense does not imply, in general, the  above condition (same example as before).
\end{rmk}

\begin{prp}\label{23.02.2017--4} The following claims about the Tannakian envelope $\Pi$ of $\Ga$ are true. 
\begin{enumerate}\item The morphism $u:\Ga\to\Pi(R)$ is dense.
\item Both fibres of $\Pi$ are reduced. 
\item
 Let $G\in(\bb{FGSch}/R)$ and let $\ph:\Ga\to G(R)$ be a dense morphism. Then, the morphism $u_\ph:\Pi\to G$  associated to $\ph$ by means of Corollary \ref{22.02.2017--4} is faithfully flat. 
 \end{enumerate}
\end{prp}

\begin{proof}
(1) and (2).  Let $B$ be the smallest $K$-subgroup scheme of $\Pi_K$ bounding (``majorant'' in French) the family $\{u(\ga)_K\}_{\ga\in\Ga}$ \cite[${\rm VI}_B$, Proposition 7.1, p.384]{SGA3}. Let $i:\ov B\to \Pi$ be the schematic closure of $B$  in $\Pi$ \cite[${\rm IV}_2$, 2.8, 33ff]{EGA}. Note that, for each $\ga\in\Ga$, the arrow $u(\ga):\mathrm{Spec}\,R\to\Pi$ factors as 
\[
\tag{$*$}u(\ga)=i(R)\circ \ps(\ga),
\] 
where $\ps(\ga):\mathrm{Spec}\,R\to\ov B$ is an $R$-point of $\ov B$. In addition, since $i(R):\ov B(R)\to\Pi(R)$ is a monomorphism, $\ps:\Ga\to\ov B(R)$ is a morphism of \emph{groups}. By Corollary \ref{22.02.2017--4}, there exists a unique morphism 
$u_\ps:\Pi\to \ov B$
such that 
\[
\ps(\ga)=u_\ps(R)\circ u(\ga)\tag{$**$}.
\] 
From ($*$) and ($**$) we deduce that $u(\ga)=(i\circ u_\ps)(R)\circ u(\ga)$, and hence $\id_\Pi= i\circ u_\ps$ (due to Corollary \ref{22.02.2017--4}). Therefore $i$ must be an isomorphism. Using \cite[${\rm VI}_B$, Proposition 7.1(ii)]{SGA3}, we conclude that $\Pi_K$ is reduced, and that the set of points associated to $\{u(\ga)_K\}_{\ga\in\Ga}$ is dense in it. In this case, it is obvious that $\{u(\ga)_K\}_{\ga\in\Ga}$ is also schematically dense \cite[${\rm IV}_3$, 11.10.4]{EGA}.

In the same spirit, let $j:H\to\Pi_k$ be the smallest $k$-subgroup  bounding    $\{u(\ga)_k\}_{\ga\in\Ga}$ and let $\ze:\Pi'\to \Pi$ be the Neron blowup of $\Pi$ at $H$. By the definition of a Neron blowup,  for each $\ga\in\Ga$ there exists $u'(\ga):\mathrm{Spec}\,R\to \Pi'$ such that 
\[
\tag{$\dagger$}u(\ga)=\ze(R)\circ u'(\ga).
\] Since $\ze(R):\Pi'(R)\to\Pi(R)$ is a monomorphism of groups, we conclude that $u':\Ga\to\Pi'(R)$ is a homomorphism. Corollary \ref{22.02.2017--4} provides $v:\Pi\to\Pi'$ such that 
\[
\tag{$\ddagger$} u'(\ga)=v(R)\circ u(\ga).
\]
As a consequence of  ($\dagger$) and ($\ddagger$), we have $u(\ga)=(\ze v)(R)\circ u(\ga)$. 
Another application of Corollary \ref{22.02.2017--4} leads to $\ze v=\id_\Pi$. Now, for each $a'\in R[\Pi']$ there exist $m\in\NN$ and $a\in R[\Pi]$ such that $\pi^ma'=\ze^\#(a)$, so that $\pi^mv^\#(a')=a$, which shows that $\pi^ma'=\pi^m\ze^\#(v^\#(a))$ thereby proving that $a'=\ze^\#(v^\#(a))$. (It is equally possible to note that $\ze(A):\Pi'(A)\to\Pi(A)$ is injective for each flat $R$-algebra $A$, and that the equation $\ze v={\rm id}_\Pi$ further implies that $\ze(A)$ is surjective.) Consequently, $\ze$ is an isomorphism and $H=\Pi_k$. As in the previous paragraph, these findings show that $\Pi_k$ is reduced and that $\{u(\ga)_k\}_{\ga\in\Ga}$ is schematically dense.


(3) Because $\ph(\ga)=u_\ph\circ u(\ga)$ for each $\ga$, it follows that the arrows   $K[G]\to K[\Pi]$ and $k[G]\to k[\Pi]$ derived from $u_\ph$ are injective. 
The  result then  follows by applying \cite[Theorem 4.1.1]{DH14}.
\end{proof}

\begin{rmk} Let $\varphi:\Ga\to \Ga'$ be a homomorphism of abstract groups. Then there exists, by the universal property, a unique group scheme homomorphism $\Phi:\Pi\to \Pi'$ rendering the following  diagram commutative
\[\xymatrix{ \Ga\ar[r]^\varphi\ar[d]_u& \Ga'\ar[d]^{u'}\\ \Pi\ar[r]_\Phi& \Pi'.}\]
\end{rmk}

\subsection{The Tannakian envelope of  an infinite cyclic group}\label{23.02.2017--7}
In this section we assume $R=\pos{k}{\pi}$, with $k$ of characteristic zero, and we give ourselves an infinite cyclic group $\Ga$ generated by $\ga$.   In the terminology of Definition \ref{22.02.2017--3}, this is certainly  a Tannakian group (apply Corollary \ref{23.02.2017--2}); let $\Pi$ stand for its Tannakian envelope over $R$. 

We wish to show that, despite the very simple nature of $\Ga$, its envelope \emph{does} allow rather ``large'' faithfully flat quotients. A precise statement is Proposition \ref{13.09.2017--1} below.

Let $G=\GG_{a,R}\ti\GG_{m,R}$ have ring of functions $R[G]=R[x,y^\pm]$. 
In the $\pi$-adic completion  $R\langle G\rangle$  of $R[G]$, we single out the element 
\[
\Ph=y-\exp(\pi x)\in R\langle G\rangle.
\]
Obviously, the closed formal subscheme  $\g H$ cut out by the ideal $(\Ph)\subset R\langle G\rangle$ is a formal subgroup, isomorphic to the $\pi$-adic completion of $\GG_{a,R}$; let $\cn$ stand for the Neron blowup of $\g H$ (see Section \ref{23.02.2017--3}). By definition, \[R[\cn]=\lid_n R[G][\pi^{-{n-1}}\Ph_n],\] 
where \[\Ph_n=y-\left(1+\pi x+\fr{\pi^2}{2!}x^2+\cdots+\fr{\pi^n}{n!}x^n\right).\]

We now introduce the ``evaluation'' at the point $P=(1,e^\pi)\in G(R)$: 
\[
\mm{ev}_P:R[G]\aro R,\quad \begin{array}{cc} x&\longmapsto 1\\ y&\longmapsto e^\pi.
\end{array} 
\]
Since $\mm{ev}_P(\Ph_n) \equiv0\mod\pi^{n+1}$,
 $P$ is a point of $\cn$. 
Let $\ph:\Ga\to\cn(R)$ send $\ga$ to $P$. According to Corollary \ref{22.02.2017--4}, there exists a morphism $u_\ph:\Pi\to \cn$
and a commutative diagram
\[
\xymatrix{\Pi(R)\ar[r]^{u_\ph(R)}& \cn(R)\\ \Ga.\ar[ru]_\ph\ar[u]^u&}
\]
\begin{prp}\label{13.09.2017--1}  The morphism $u_\ph:\Pi\to\cn$ is faithfully flat. 
\end{prp}

\begin{proof} According to \cite[Corollary 5.11]{DHdS15},  $\cn\ot k$ is just $\g H\ot k\simeq\GG_{a,k}$; as the image of $P$ in $\cn(k)$ is not the neutral element, the assumption on the characteristic  proves that  the subgroup it generates is schematically dense. On the generic fibre, the Zariski closure $C$ of the subgroup generated by $P$ must be of dimension at least two since both restrictions $C\to \GG_{m,K}$ and $C\to\GG_{a,K}$ are faithfully flat (the standard facts employed here are in \cite{waterhouse}, see the Corollary on p.65 and the Theorem on the bottom of p.114). This implies that $C=G_K$ and 
Proposition \ref{23.02.2017--4} then finishes the proof.  
\end{proof}

\begin{cor}\label{12.10.2018--2}
The group scheme $\Pi$ is not strictly pro-algebraic (in the sense of  Question (\textbf{SA}) from the Introduction). In fact, the Tannakian envelope $\Pi^\star$ of any abstract group $\Ga^\star$ admitting $\Ga$ as a quotient fails to be strictly pro-algebraic. 
\end{cor}
\begin{proof}We assume that $\Pi$ is strictly pro-algebraic and write $R[\Pi]=\lid_iR[\cg_i]$, where each $\cg_i\in(\bb{FGSch}/R)$ is of finite type, and each arrow $R[\cg_i]\to R[\cg_j]$ is faithfully flat. In particular, each arrow $R[\cg_i]\to R[\Pi]$ is faithfully flat (since tensor products commute with direct limits), and 
$R[\cg_i]$ is saturated in both $R[\cg_j]$ and  $R[\Pi]$. Let $i_0$ be such that $x,y,y^{-1}\in R[\cg_{i_0}]$ (we identify $R[\cn]$ with its image in $R[\Pi]$). As $K[\cn]$ is generated by $\{x,y,y^{-1}\}$ and $R[\cg_{i_0}]$ is saturated in $R[\Pi]$, we arrive at a commutative diagram 
 factoring $u_\ph$:
\[
\xymatrix{&\cg_{i_0}\ar[dr]^{\chi}&  \\ \ar[ru]^q\Pi\ar[rr]_{u_\ph}&&\cn.}
\]
It then follows, since $u_\ph$ and $q$ are faithfully flat, that $\chi$ is faithfully flat \cite[p. 46]{matsumura}. 
Together with Exercise 7.9 on p. 53 of \cite{matsumura}, this implies that $R[\cn]$ is a noetherian ring, something which is certainly false. (To verify this last claim, one proceeds as follows. Let $\g a$ be the augmentation ideal of $R[\cn]$, and let $\Ps_n=\pi^{-n-1}\Ph_n$. Then   $\pi\Ps_{n+1}\equiv\Ps_n\mod\g a^2$ for all $n\ge1$, so that  $\g a/\g a^2$ is either not finitely generated over $R=R[\cn]/\g a$, or  $y-1-\pi x\in\g a^2$. But this latter condition is impossible.)

To verify the second statement, we observe that if  $\ps:\Ga^\star\to\Ga$ is an epimorphism, then the induced arrow $\Ps:\Pi^\star\to\Pi$ is faithfully flat due to \cite[Proposition 3.2.1(ii)]{DH14}. The exact same proof as above now proves the second statement. 
\end{proof}

\section{Neron blowups of formal subgroup schemes in differential Galois theory}\label{24.02.2017--1}

In this section, we set out  to investigate whether blowups of closed formal subgroup schemes, as in Section \ref{23.02.2017--3}, appear in differential Galois theory over $R$. (This theory is to be understood as the one explained in \cite[Section 7]{DHdS15}.) Our strategy is to restrict attention to the discrete valuation ring  $\pos{\CC}{t}$. We shall see that through a ``Riemann-Hilbert correspondence'', which is Theorem \ref{Riemann-Hilbert} below, we are  able to transplant the results of Section \ref{23.02.2017--7} concerning the category of representations of an abstract group to swiftly arrive at a conclusion: see Corollary \ref{23.02.2017--6}. The proof of Theorem \ref{Riemann-Hilbert} is routine and employs Serre's GAGA,  Grothendieck's Existence Theorem in formal Geometry (GFGA), and Deligne's  dictionary \cite[I.2]{Del70}. 


\subsection{Preliminaries on $\cd$-modules}\label{10.02.2017--2}
We now review some useful notions concerning $\cd$-modules.  These are employed to establish an algebraization result, Proposition \ref{03.07.2017--1}, further ahead.


Let $T$ be a noetherian scheme and $Y$  a smooth $T$-scheme; we write $\cd(Y/T)$ for the ring of differential operators described in \cite[${\rm IV}_4$, 16.8]{EGA} and \cite[\S2]{BO}. Following the notation of \cite[\S7]{DHdS15}, we let $\modules{\cd(Y/T)}$ stand for the category of (sheaves of) $\cd(Y/T)$-modules on $Y$ whose underlying $\co_Y$-module is coherent. If $T$ is not a $\QQ$-scheme, then the ring $\cd(Y/T)$ is usually not finitely generated over $\co_Y$, so that a fundamental tool in dealing with $\cd$-modules is Grothendieck's notion of stratification. 

Let $\cp_{Y/T}^\nu$ stand for the sheaf of principal parts of order $\nu$ of $Y\to T$ \cite[${\rm IV}_4$, 16.3.1, p.14]{EGA}. 
The scheme $(|Y|,\cp_{Y/T}^\nu)$ \cite[${\rm IV}_4$, 16.1]{EGA}, usually named the \emph{scheme of principal parts of order $\nu$}, is denoted by $P_{Y/T}^\nu$. It comes with two projections $p_0,p_1:P_{Y/T}^\nu\to Y$ and a closed immersion $\De:Y\to P_{Y/T}^\nu$, and we follow the convention of EGA fixing the first projection as the structural one, while the other is supplementary. 
 
Given $\ce\in\modules{\cd(Y/T)}$, for any $\nu\in\NN$ we can associate an isomorphism of $\co_{P^\nu_{Y/T}}$-modules 
\[\te(\nu):p_1^*(\ce)\arou\sim p_0^*(\ce)\]
which, when pulled-back along $\De$ gives back the identity. These are are called {\it Taylor isomorphisms}  of order $\nu$.
 Clearly, using the obvious closed immersions $P_{Y/T}^{\nu}\to P_{Y/T}^{\nu+1}$, we obtain the notion of a  compatible family of Taylor isomorphisms.

These considerations prompt the definition of a pseudo-stratification of a coherent module $\cf$:  It is a \emph{compatible} family of isomorphisms  
\[
\si(\nu):p_1^*(\cf)\arou\sim p_0^*(\cf),\quad\nu=0,\ldots ,
\]
which induce the identity when pulled-back along the diagonal. The pseudo-stratification $\{\si(\nu)\}_{\nu\in\NN}$ is a stratification if the isomorphisms satisfy, in addition, a cocycle condition  \cite[Definition 2.10]{BO}.  Coherent modules endowed with stratifications constitute a category, denoted by  $\mathbf{str}(Y/T)$. 
The main point of these definitions is that the Taylor isomorphisms define an equivalence of categories   $\modules{\cd(Y/T)}\to\mathbf{str}(Y/T)$ which preserves all the underlying coherent modules and morphisms   \cite[2.11]{BO}.

Although stratifications are a more intricate version of the   natural concept of $\cd$-module, their  functoriality properties are most welcome. From a commutative diagram of schemes  
\[
\xymatrix{Z\ar[r]^\ph\ar[d]&\ar[d]Y\\ 
U\ar[r] & T,}
\] 
we obtain another such diagram  
\[
\xymatrix{P_{Z}^\nu\ar[d]_{p_{i,Z}}\ar[rr]^{P^\nu(\ph)} && P_{Y}^\nu\ar[d]^{p_{i,Y}}\\
\ar[d]  Z\ar[rr]_{\ph}  && Y\ar[d]\\  U\ar[rr]&& T,}
\]
where $P^\nu(\ph)$ is the morphism defined by $|\ph|:|Z|\to|Y|$ and the arrow of $\co_Z$-algebras (abusively denoted by the same symbol)
\[
P^\nu(\ph):\co_Z\otu{\co_Y}\cp^\nu_{Y/T}\aro \cp^\nu_{Z/U},
\] which is explained on  \cite[${\rm IV}_4$, 16.4.3, 17ff]{EGA}.
Consequently, we have an isomorphism of functors 
\[
P^\nu(\ph)^*\circ p_{i,Y}^*\simeq p_{i,Z}^*\circ\ph^*,
\]
which allows us to prolong $\ph^*:\bb{Coh}(Y)\to\bb{Coh}(Z)$ into   
\[
\ph^\#:\bb{str}(Y/T)\aro\bb{str}(Z/U),\quad(\ce,\{\te(\nu)\})\longmapsto(\ph^*\ce,\{\ph^\#\te(\nu)\} );
\]
just define  $\ph^\#\te(\nu)$ by decreeing that 
\[
\xymatrix{P^\nu(\ph)^*p_{1,Y}^*(\ce)\ar[d]_{\sim}\ar[rr]^{P^\nu(\ph)^*\te(\nu)}&&P^\nu(\ph)^*p_{0,Y}^*(\ce)\ar[d]^{\sim}\\
p_{1,Z}^*\ph^*(\ce)\ar[rr]_{\ph^\# \te(\nu)}&& p_{0,Z}^*\ph^*(\ce)
}
\]
commutes. Using the equivalence $\modules{\cd(Y/T)}\simeq {\bf str}(Y/T)$ mentioned above, we can also prolong $\ph^*$ to a functor $\ph^\#:\modules {\cd(Y/T)}\to\modules{\cd(Z/U)}$.

\begin{rmk}
Similar notions hold  for the case where $g:Y\to T$ is replaced by   a smooth morphism of \emph{complex analytic spaces} \cite[I.2.22]{Del70}. This being so, we adopt similar notations.   
\end{rmk}

\subsection{An algebraization result}\label{20.09.2017--2} In this section, we assume that $R$ is complete. Let \[f:X\aro S\] be a smooth morphism (recall that $S={\rm Spec}\,R$). For every  $n\in\NN$, we obtain  cartesian commutative diagrams 
\[
\xymatrix
{
X_{n}\ar[d]\ar[r]^-{u_{n}}&\ar[d]X_{n+1}\\
S_{n}\ar[r]&S_{n+1}}\]
and
\[
\xymatrix
{
X_{n}\ar[d]\ar[r]^-{v_{n}}&\ar[d]X\\
S_{n}\ar[r]&S, 
}
\]
where the vertical arrows are the evident closed embeddings. In this case, the arrows  
\begin{equation}\label{19.09.2017--1}
P^\nu(v_n): \co_{X_n}\otu{\co_X}\cp_{X/S}^\nu\aro\cp^\nu_{X_n/S_n}
\end{equation}
are isomorphisms \cite[${\rm IV}_4$, 16.4.5]{EGA}.  Hence, the \emph{dual} arrows 
\[
 \cd(X_n/S_n)\aro \co_{X_n}\otu{\co_X}\cd(X/S)
\]
are isomorphisms of $\co_{X_n}$-modules (here we used that the modules of principal parts are locally free \cite[Proposition 2.6]{BO}). This  particular situation allows us to construct an arrow of  $\co_X$-modules  
\begin{equation}\label{19.09.2017--2}
\rho_n:\cd(X/S)\aro\cd(X_n/S_n).
\end{equation}
A moment's thought proves that for any differential operator (on some open of $X$) $\pd:\co_V\to\co_V$, the operator $\rho_n(\pd):\co_{V_n}\to\co_{V_n}$ is just the reduction of $\pd$ modulo $\pi^{n+1}$, so that $\rho_n$ is a morphisms of $v_n^{-1}\co_X$-algebras. Moreover, $\rho_n$ allows us to identify $\cd(X/S)/(\pi^{n+1})$ and  $\cd(X_n/S_n)$ in such a way that the functor $v_n^\#:\modules{\cd(X/S)}\to\modules{\cd(X_n/S_n)}$ is  the obvious one.

\begin{dfn}\label{20.09.2017--1}Denote by $\bb{Coh}^\wedge(X)$ the category constructed as follows. Objects are sequences $(\ce_n,\al_n)$, where  $\ce_n\in\bb{Coh}(X_n)$, and  $\al_n:u_n^*\ce_{n+1}\to\ce_n$ an isomorphism. Morphisms between $(\ce_n,\al_n)$ and $(\cf_n,\be_n)$ are families $\ph_n:\ce_n\to\cf_n$ fulfilling $\ph_n\circ\al_n=\be_n\circ u_n^*(\ph_n)$, see \cite[8.1.4]{Illusie}. (Needless to say, this is the category of coherent modules on the formal completion.)
We let $\modules{\cd(X/S)}^\wedge$ be constructed analogously from $\modules{\cd(X/S)}$ and the functors $u_n^\#$. 
\end{dfn} 

Let 
\[ \bb{Coh}(X)\arou\Ph\bb{Coh}(X)^\wedge\quad \text{and} \quad\modules{\cd(X/S)}\arou{\mm D\Ph}\modules{\cd(X/S)}^\wedge\] be the obvious functors.
Grothendieck's existence theorem \cite[Theorem 8.4.2]{Illusie}, henceforth called GFGA, says that \emph{if $f$ is proper}, then $\Ph$ is an equivalence of $R$-linear $\ot$-categories. Using the notion of stratification, we have:

\begin{prp}\label{03.07.2017--1}Assume that $f:X\to S$ is proper. Then, $\mm D\Ph$ 
is an equivalence.  
\end{prp}
 
\begin{proof}One obvious proof is to apply GFGA to the  isomorphisms of coherent sheaves on the proper $S$-schemes $P^\nu_{X/S}$ (the Taylor isomorphisms) and then prove that these define a stratification. 
For the sake of clarity, we choose to avoid the second step, so that the number of verifications relying on canonical identifications is reduced. In taking this approach, we shall require Lemma \ref{14.09.2017--1} below; its  proof is simple and we leave it to the reader.

\emph{Essential surjectivity:} Due to GFGA, we need to prove the following
\begin{claim*}  Let $\ce\in\bb{Coh}(X)$, and assume that (i) each $\ce_n:=\co_{X_n}\ot \ce$ has the structure of a $\cd(X_n/S_n)$-module, and (ii) the canonical isomorphisms $\al_n:u_n^*\ce_{n+1}\to\ce_n$ preserve these structures. Then $\ce$ carries a $\cd(X/S)$-module structure, and $\mm D\Ph(\ce)=(\ce_n,\al_n)$.

\end{claim*}

Let $\nu\in\NN$ be fixed, and for each $n\in\NN$, write 
\[
\te(\nu)_{n}:p_{1,X_n}^*(\ce_n)\aro p_{0,X_n}^*(\ce_n)
\] 
for the Taylor isomorphism of order $\nu$ associated to the $\cd(X_n/S_n)$-module $\ce_n$. The fact that $\al_n$ is an isomorphism of stratified modules implies,  after the necessary identifications, that 
\begin{equation}\label{26.01.2018--2}
P^\nu(u_n)^*(\te(\nu)_{n+1})=\te(\nu)_n,\quad\text{ for all $n$.}
\end{equation}                                                                                                                                                                                                                                                                                                                                                                                                                                                                                                                                                                                                                                                                                                                  By the isomorphism \eqref{19.09.2017--1}, we can  apply GFGA  to the proper $S$-scheme   $P_X^\nu$ to obtain an isomorphism of $\co_{P^\nu_X}$-modules
\[
\te(\nu):p_{1,X}^*(\ce)\aro p_{0,X}^*(\ce)\]
such that 
\[
P^\nu(v_n)^*(\te(\nu))=\te(\nu)_n\quad\text{for all $n$.}
\]

Let $V\subset X$ be an open subset constructed from $\ce$ as in Lemma \ref{14.09.2017--1} below, and let $\pd\in \cd(X/S)(V)$ be of order $\le\nu$. Let $e\in\ce(V)$. We define 
\[\na_{\nu}(\pd)(e):= (\id\ot\pd)\left[\te(\nu)(1\ot e)\right],\] 
where we remind the reader that $p_{0,X}^*(\ce)=\ce\ot \cp_{X}^\nu$ and  $p_{1,X}^*(\ce)=   \cp_{X}^\nu\ot\ce$ as sheaves on $|X|$. At this point, $\na_{\nu}(\pd)(e)$ is just a name for a section of $\ce(V)$. We have $\co(V)$-linear maps  $q_n:\ce(V)\to\ce_n(V\cap X_0)$ and $\rho_n$ (as in eq. \eqref{19.09.2017--2})   and it is clear from the definition and eq. \eqref{26.01.2018--2} that 
\begin{equation}\label{26.01.2018--1}
q_n[\na_{\nu}(\pd)(e)]=\rho_n(\pd)\po q_n(e),\qquad\text{for each $n$.}
\end{equation}
As $(q_n):\ce(V)\to\lip_n\ce_n(V\cap X_0)$ is injective (Lemma \ref{14.09.2017--1}-(b)), we can say:
\begin{enumerate}\item  if $\nu\le\nu'$, then $\na_{\nu}(\pd)(e)=\na_{\nu'}(\pd)(e)$. 
\item If $\pd'$   is such that $\pd\pd'$ has order  $\le\nu$, then $\na_{\nu}(\pd\pd')(e)=\na_{\nu}(\pd)[\na_{\nu}(\pd')(e)]$.
\end{enumerate}
This means that $\ce(V)$ now carries the structure of a $\cd(X/S)(V)$-module.  Because     $V$    is affine, this endows $\ce|_V$ with a structure of $\cd(V/S)$-module. Furthermore, by definition, the Taylor isomorphism of level $\nu$ associated to this structure is simply the restriction of $\te(\nu)$ to $V$. Hence, covering $X$ by open subsets as these,  $\ce$ becomes a  $\cd$-module for which the Taylor isomorphism of order $\nu$ is simply $\te(\nu)$. By construction (see eq. \eqref{26.01.2018--2}) we then conclude that $\mm D\Ph(\ce)=(\ce_n,\al_n)$.

\emph{Fully faithfulness:}  As before, by GFGA, it is sufficient to show 

\begin{claim*} Let $\ce$ and $\cf$ be $\cd(X/S)$-modules, and let  $\phi:\ce\to\cf$ be an $\co_X$-linear arrow  such that $\phi_n:v_n^*\ce\to v_n^*\cf$ is $\cd(X_n/S_n)$-linear for all $n$. Then  $\phi$ is a morphism of $\cd(X/S)$-modules. 
\end{claim*}
Let $V$ be constructed from $\cf$ as in Lemma \ref{14.09.2017--1}(a); we show that  $\phi|_V$ is $\cd(V/S)$-linear. 
Let $e\in\ce(V)$, $\pd\in\cd(X/S)(V)$, and write $r_n$  for the natural $\co(V)$-linear morphisms  $\cf(V)\to v_n^* \cf(V\cap X_0)$. It follows easily that $r_n(\pd\po\phi(e))=r_n(\phi(\pd \po e))$, so that   injectivity of $(r_n):\cf(V)\to\lip v_n^*(\cf)(V\cap X_0)$
assures that $\phi|_V$ is $\cd(V/S)$-linear.  Since $X$ can be covered by open subsets like $V$ (Lemma \ref{14.09.2017--1}), we are done. 
\end{proof} 
 
\begin{lem}\label{14.09.2017--1}Let $\cm$ be a coherent $\co_X$-module. 

\begin{enumerate}\item[(a)] Denote by $\xi_1,\ldots,\xi_r$ the associated points of $\cm$ \cite[${\rm IV}_2$, 3.1]{EGA} lying on the generic fibre $X_K$. Then, every point $x\in X_k$ has an open affine coordinate neighbourhood $V$ such that if $\xi_j\in V$, then $\ov{\{\xi_j\}}\cap V_k\not=\varnothing$. 

\item[(b)] If $V$ is as before, then $\cm(V)$ is $\pi$-adically separated.

\item[(c)] The scheme $X$ can be covered by finitely many open subsets as $V$ above. 
\qed 
\end{enumerate}
\end{lem}


\subsection{Local systems and representations of the topological fundamental group}\label{17.03.2017--1}
Let $\La$ be a local $\CC$-algebra which, as a complex vector space, is finite dimensional. (In this case, the residue field of $\La$ is $\CC$.) Writing  $T$ for the  analytic space associated to $\La$ \cite[Expos\'e 9, 2.8]{sc}, we give ourselves   a smooth morphism $g:Y\to T$ of \emph{connected} complex analytic spaces \cite[Expos\'e 13, \S3]{sc}.

Let us now bring to stage the category $\bb{LS}(Y/T)$ of relative local systems \cite[I.2.22]{Del70}.
These are $g^{-1}(\co_T)$-modules $\cf$ on $Y$ enjoying the ensuing property: for every $y\in Y$, there exists an  open  $V\ni y$  such that the $g^{-1}(\co_T)$-module $\cf|V$ belongs to $g^{-1}\bb{Coh}(T)$. 
Under these conditions, any $\cf\in\bb{LS}(Y/T)$ is 
a \emph{locally constant} sheaf of $\La$-modules whose values on connected open sets are of finite type. Conversely, any such locally constant sheaf of $\La$-modules satisfies the defining property of $\bb{LS}(Y/T)$. Hence, the following arguments constitute a well-known exercise, see Exercise F of Ch. 1 and Exercise F of Ch. 6 in \cite{spanier}. (Here one should recall that $|Y|$ is the topological space of a {\it manifold}.) 

Let $y_0\in |Y|$ be a point and $\Ga$ the fundamental group of $|Y|$ based at it. Given a locally constant sheaf of $\La$-modules $\cf$ on $Y$, we can endow its stalk $\cf_{y_0}$ with a left action of $\Ga$. This construction then allows the definition of a functor 
\[
\mathds M_{y_0}: \left\{\begin{array}{c}\text{locally constant sheaves}\\ \text{ of  $\La$-modules whose }\\ \text{stalks are of finite type}\end{array}\right\}\aro {\rm Rep}_\La(\Ga). 
\]


 
\begin{prp}\label{17.03.2017--3}The functor $\mathds M_{y_0}$ defines a $\La$-linear $\ot$-equivalence between $\bb{LS}(Y/T)$ and $\rep{\La}{\Ga}$. 
\qed 
\end{prp}

\subsection{$\cd$-modules on the analytification of an algebraic $\CC$-scheme}
Let $\La$ be a local $\CC$-algebra which, as a complex vector space, is finite dimensional and write $T=\mathrm{Spec}\,\La$. 
We give ourselves a smooth morphism of \emph{algebraic  $\CC$-schemes}  $g:Y\to T$ and set out to explain briefly the equivalence of $\cd$-modules on $Y$ and on its analytification, see Proposition \ref{17.03.2017--4} below. (We observe that these constructions are routine and often left aside, see pp. 97-8 in \cite{Del70}, for example.)

Let 
\[
\xymatrix{Y^\an\ar[r]^\al\ar[d]_{g^\an}& Y\ar[d]^g\\ T^\an\ar[r]&T}
\] 
stand for the diagram of locally ringed spaces obtained from the analytification functor \cite[XII, 1.2]{SGA1}. 
Proceeding as in \ega{IV}{4}{16.4}, while suppressing reference to $\al^{-1}$, $T$ and $T^\an$, we arrive at a morphism of $\co_{Y}$-modules 
\[\{\al\}_\nu:\cp^\nu_Y\aro\cp^\nu_{Y^\an} 
\]
such that, if $d^\nu$ is as in \ega{IV}{4}{16.3.6, p.15}, then 
$\{\al\}_\nu(  d^\nu(f))=d^\nu(\al^\#(f))$.   
(The justification given on page 17 of \ega{IV}{4}{16.4.3} for the existence of $\{\al\}_\nu$ relies on the case of affine schemes, but  a proof based only on formal properties of normal invariants is possible.) We then have a morphism of $\co_{Y^\an}$-modules 
\begin{equation}\label{12.10.2018--1}
P^\nu(\al):\co_{Y^\an}\otu{\co_Y}\cp_{Y}^\nu \aro \cp_{Y^\an}^\nu 
\end{equation}
such that 
\begin{equation}\label{10.10.2018--1}
P^\nu(\al)(1\ot d^\nu(f))=d^\nu(\al^\#(f)).
\end{equation} 
If $V\subset Y$ is an open subset having etale coordinates $\{y_j\}$ over $T$, then $\{d^\nu( y^q)\,:\,\sum q_j\le \nu\}$ is a basis of $\cp^\nu_Y|_V$ \cite[Proposition 2.6]{BO}. Similarly, as $V^\an$ is locally isomorphic to an open set of an affine space over $T^\an$ \cite[Expos\'e 13, 3.1]{sc}, the set
$\{d^\nu( \al^\#(y)^q)\,:\,\sum q_j\le \nu\}$ is a basis of $\cp^\nu_{Y^\an}|_{V^\an}$;  we conclude with the help of equation \eqref{10.10.2018--1} that $P^\nu(\al)$ is  an isomorphism. In particular, 
$(\cp^\nu_Y)^\an\stackrel\sim\to\cp^\nu_{Y^\an}$ via $P^\nu(\al)$ so that  $(|Y^\an|,(\cp^\nu_Y)^\an)\simeq P_{Y^\an}^\nu$. Now we note that  $(|Y^\an|,(\cp_Y^\nu)^\an)$ is just $(P^{\nu}_Y)^\an$.

Taking the dual of arrow \eqref{12.10.2018--1} for varying $\nu$, we obtain an {\it isomorphism} of $\co_{Y^\an}$-modules  
$\cd(Y^\an)\to\co_{Y^\an}\ot_{\co_Y}\cd(Y)$
which allows us to construct an arrow of $\co_Y$-modules 
\[
\rho:\cd(Y)\aro\cd(Y^\an).
\]
Using $V$ and $\{y_j\}$ as above and equation \eqref{10.10.2018--1}, the  verification  that $\rho$ is a morphism of $\co_Y$-\emph{algebras} is straightforward.

If $\ce\in \modules{\cd_Y}$ and $\{\te(\nu)\}$ are its Taylor isomorphisms (see Section \ref{10.02.2017--2}), we arrive at a commutative diagram of $\cp^\nu_{Y^\an}$-modules 
\begin{equation}\label{16.10.2018--3}
\xymatrix{
\cp^\nu_{Y^\an}\otu{ \cp^\nu_Y} \left(\cp^\nu_Y\otu{\co_Y}\ce\right) \ar[rrr]^-{\id\ot\te(\nu)}_-{\sim}\ar@{=}[d]
&&& 
\cp^\nu_{Y^\an}\otu{\cp^\nu_Y}\left(\ce\otu{\co_Y}\cp^\nu_Y\right)\ar@{=}[d]
\\
\cp^\nu_{Y^\an}\otu{\co_{Y^\an}}\ce^\an\ar[rrr]_-{\sim} &&& \ce^\an\otu{\co_{Y^\an}}\cp^\nu_{Y^\an}
}
\end{equation}
allowing us to define an   
analytification functor: 
\[(-)^\an:\modules{\cd(Y/T)} \aro\modules{\cd(Y^\an/T^\an)}.
\]
An argument similar to the one used in proving Proposition \ref{03.07.2017--1} gives: 

\begin{prp}\label{17.03.2017--4} Suppose that $g$ is proper. Then the functor $(-)^\an:\modules{\cd(Y/T)}\to\modules{\cd(Y^\an/T^\an)}$
is an equivalence of $\La$-linear  $\ot$-categories.
\end{prp}
\begin{proof}[Sketch of proof.] We deal only with essential surjectivity. By  GAGA \cite[XII, Theorem 4.4]{SGA1}, we only need to show that if $\ce$ is a coherent $\co_{Y}$-module such that $\ce^\an$ carries a structure of $\cd(Y^\an)$-module, then $\ce$ carries a structure of $\cd(Y)$-module inducing the preceding one. So let $\nu\in\NN$ be fixed and write $\si(\nu):p^{\an*}_{1}\ce^\an\to p_0^{\an*}\ce^\an$ for the Taylor isomorphism of order $\nu$. 
Applying GAGA  to the proper $\CC$-scheme  $P_{Y}^\nu=(|Y|,\cp_Y^\nu)$ we see that  $\si(\nu)$ comes from an isomorphism $\te(\nu):p_1^*\ce\to p_0^*\ce$ of $\cp_{Y}^\nu$-modules. Let $V$ be an affine and open subscheme of $Y$ having etale coordinates $\{y_j\}$ as above, $e$ a section of $\ce$ over $V$,  and $\pd\in\cd(Y)(V)$ of order $\le\nu$. Define $\na_\nu(\pd)(e)=({\rm id}_\ce\ot\pd)[\te(\nu)(1\ot e)]$, which is a section of $\ce$. If $(-)^\an:\Ga(V,\ce)\to\Ga(V^\an,\ce^\an)$ denotes the canonical map, then equation \eqref{10.10.2018--1} and  diagram \eqref{16.10.2018--3} show that $\rho(\pd)(e^\an)=(\na_\nu(\pd)(e))^\an$. This allows us to define on $\ce$ the structure of a $\cd(Y)$-module having $\{\te(\nu)\}$ as Taylor isomorphisms; details are simple and similar to those in the proof of Proposition \ref{03.07.2017--1}.  \end{proof}

To complement this Proposition, we note the following. Let $y$ be a $T$-point  of $Y$ sending $|T|$ to $y_0$. If $y^\an$ is the associated  $T^\an$-point of $Y^\an$, then it sends $|T^\an|$ to $y_0$, and, for an object $\ce$ of $\modules{\cd(Y/T)}$, the natural morphism of stalks  $\ce_{y_0}\to(\ce^\an)_{y_0}$ defines an isomorphism $y^*\ce\to y^{\an*}\ce^\an$.
This produces a natural isomorphism  $y^*\stackrel{\sim}{\to} y^{\an*}\circ(-)^\an$, where we abuse notation and identify $\bb{Coh}(T)$ and $\bb{Coh}(T^\an)$.  

\subsection{$\cd$-modules and relative local systems}Notations now are those of Section \ref{17.03.2017--1}, so that we are concerned only with complex analytic spaces.

For each $\ce\in\modules{\cd(Y/T)}$, we consider the sheaf of $\La$-modules (``solutions'', ``horizontal sections'') 
\[
{\bf Sol}(\ce)(U):=\left\{s\in\ce(U)\,:\,\begin{array}{c} \text{ for  $V\subset U$ open, the section $s|_V$ is annihilated }\\\text{  by the augmentation ideal of $\cd(Y/T)(V)$}
\end{array}\right\}.
\]
Now, recalling that $\cd(Y/T)$-modules are exactly the same thing as integrable connections \cite[Theorem 2.15]{BO},  Deligne shows in  \cite[I.2.23]{Del70} that: 
\begin{thm}\label{17.03.2017--2}For each $\ce\in\modules{\cd(Y/T)}$, the sheaf of $\La$-modules ${\bf Sol}(\ce)$ is a relative local system and the morphism of $\co_Y$-modules \[\co_Y\ot_\La{\bf Sol}(\ce)\aro\ce\] is an isomorphism. In addition, the functor 
\[
{\bf Sol}:\modules{\cd(Y/T)}\aro \bb{LS}(Y/T) 
\]
is a $\ot$-equivalence of $\La$-linear   categories. 
\end{thm}

Putting together Proposition \ref{17.03.2017--3} and  Theorem \ref{17.03.2017--2}, we have:

\begin{cor}\label{12.04.2017--1}  The composite functor 
\[\modules{\cd(Y/T)}\arou{{\bf Sol}}  \bb{LS}(Y/T)  \arou{\mathds M_{y_0}} \rep{\La}{\Ga}
\]
defines a $\ot$-equivalence of $\La$-linear categories.  \qed
\end{cor}

To complement this corollary, we make the following observations. Let $y:T\to Y$ be a section to the structural morphism which takes $|T|$ to $y_0$, and let $\ce$ be an object of $\modules{\cd(Y/T)}$. On the level of stalks at $y_0$, Theorem \ref{17.03.2017--2} implies that the natural  morphism $\co_{y_0}\ot_\La\bb{Sol}(\ce)_{y_0}\to \ce_{y_0}$ is an isomorphism, so that the composite arrow $\bb{Sol}(\ce)_{y_0}\to \ce_{y_0}\to  y^*(\ce)$ is an isomorphism. 
Hence, if $\om:\rep\La\Ga\to\modules\La$ stands for the forgetful functor, it follows that $\om\circ\mathds M_{y_0}\circ\bb{Sol}\stackrel\sim\to y^*$ as $\ot$-functors.   

\subsection{The category $\mm{Rep}_R(\Ga)^\wedge$} 
In this section we shall assume that  $R$ is complete and denote by $t$ a uniformizer.  
Given an abstract group  $\Ga$, we introduce the category  $\mm{Rep}_R(\Ga)^\wedge$ following the pattern of Definition \ref{20.09.2017--1}.   Its objects are sequences $(E_n,\si_n)_{n\in\NN}$ where each $E_n$ is a left ${R_n}\Ga$-module whose underlying  $R_n$-module is finite, and 
\[
\si_n:R_{n}\ot_{R_{n+1}}E_{n+1}\aro E_n
\] 
is an isomorphism of $R_n\Ga$-modules. Term by term tensor product (of $R_n$-modules) defines on $\rep R\Ga^\wedge$ the structure of a $R$-linear $\ot$-category. 

The categories $\rep R\Ga$ and $\rep R\Ga^\wedge$ are related by two functors 
\[\g F:\rep R\Ga\aro \rep R\Ga^\wedge\]
and 
\[
\g L:\mm{Rep}_R(\Ga)^\wedge\aro \mm{Rep}_R(\Ga).
\]
On the level of modules these are defined by 
\[\g F(E)=(R_n\ot_R E,\mm{canonic})\quad\text{and}\quad\g L(E_n,\si_n)=\lip_nE_n.\]
(That the projective limit above is of finite type over $R$ is proved in \ega{0}{I}{7.2.9}.)  Needless to say, $\g L$ and $\g F$ are inverse $\ot$-equivalences   of $R$-linear categories, see 7.2.9-11 and 7.7.1 in \cite[$0_{\rm I}$]{EGA}. 

\begin{rmk}
The additive category ${\rm Rep}_R(\Ga)^\wedge$ is abelian and $\g F$ respects that structure, but kernels   \emph{are not} defined effortlessly because for an arrow $(\ph_n):(E_n)\to(F_n)$, the projective system of kernels $(\mathrm{Ker}\,\ph_n)$ is not necessarily an object of ${\rm Rep}_R(\Ga)^\wedge$. For example, one can consider the obvious object $(R_n)$ of $\rep R\Ga^\wedge$ and  the arrow $(\ph_n):(R_n)\to(R_n)$ of multiplication by $t$; then the kernel is certainly trivial, while  $\mm{Ker}(\ph_0)=R_0$. 
\end{rmk}

\subsection{The main result}\label{04.07.2017--4}We suppose here that  $R=\pos\CC t$ and return to the setting of Section \ref{20.09.2017--2}, where, among others, we considered a \emph{smooth} morphism of schemes
\[
f:X\aro S.
\]
We assume in addition that $f$ is \emph{proper} and that  $X_0$ is \emph{connected}. Note that, the latter condition implies that  each $|X_n^\an|$ is also connected \cite[XII Proposition 2.4]{SGA1}.

Following the model of Definition \ref{20.09.2017--1}, we use the $\NN$-indexed family of categories $\modules{\cd(X_n^\an/S_n^\an)}$
to produce 
\[
\modules{\cd(X^\an/S^\an)}^\wedge.
\] 
(Here  $X^\an$ and $S^\an$ are just graphical signs carrying no mathematical meaning.)
Since $X_n^\an$ is proper over  $S_n^\an$ \cite[XII, Proposition 3.2, p. 245]{SGA1}, 
Proposition \ref{17.03.2017--4} produces   $\ot$-equivalences of $R_n$-linear  categories 
\[
{(-)^\an}:\modules{\cd(X_n/S_n)}\arou\sim \modules{\cd(X_n^\an/S_n^\an)},
\]
which in turn give rise to a $\ot$-equivalence of $R$-linear categories 
\[
\g A:\modules{\cd(X/S)}^\wedge \arou\sim  \modules{\cd(X^\an/S^\an)}^\wedge.
\]

Let $x$ be an $R$-point of $X$ and write $x_n$ for the corresponding $R_n$-point of $X_n$. Clearly, each $x_n$ induces the same point, call it $x_0$, on the topological space  $|X_0^\an|=|X_n^\an|$. Let  $\Ga$ be the fundamental group of $|X_0^\an|$ based at $x_0$. From Corollary \ref{12.04.2017--1} we have  $\ot$-equivalences of $R_n$-linear categories: 
\[
\mathds M_{x_0}\circ{\bf Sol}:\modules{\cd(X_n^\an/S_n^\an)}\arou\sim \rep{R_n} \Ga,
\]
which in turn induce an equivalence 
\[
\g M_{x_0}:  \modules{\cd(X^\an/S^\an)}^\wedge\arou\sim      {\rm Rep}_R(\Ga)^\wedge. 
\]
We have therefore a diagram 
\begin{equation}\label{13.04.2017--1}
\xymatrix{
\modules{\cd(X/S)}\ar[r]^-{\mm D\Ph}\ar[r]& \modules{\cd(X/S)}^\wedge \ar[r]^-{\g A}&
\modules{\cd(X^\an/S^\an)}^\wedge\ar[r]^-{\g M_{x_0}}&
\\
\ar[r]^-{\g M_{x_0}} & \rep R\Ga^\wedge\ar[r]^{\g L}& \rep{R}\Ga,&}
\end{equation}
where each arrow is a $\ot$-equivalence of $R$-linear  categories. We can therefore say: 
\begin{thm}\label{Riemann-Hilbert} The $R$-linear $\ot$-categories $\modules{\cd(X/S)}$ and $\rep{R}{\Ga}$ are equivalent by means of the composition of \eqref{13.04.2017--1} above. 
\qed 
\end{thm}

Since we are also interested in the differential Galois groups associated to objects in $\modules{\cd(X/S)}$ (see  \cite[section 7]{DHdS15} and Section \ref{22.08.2020--1}), we complement Theorem \ref{Riemann-Hilbert} by explaining how the aforementioned equivalence relates fibre functors.  As we observed after Proposition \ref{17.03.2017--4} and Corollary \ref{12.04.2017--1}, $x_n^*\simeq x_n^{\an*}\circ(-)^\an$ and $x_n^{\an*}\simeq \om\mathds M_{x_0}\bb{Sol}$. This implies that $\om\g M_{x_0}\g A$ is $\ot$-isomorphic to $(\ce_n)\mapsto(x_n^*\ce_n)$ so under the above equivalence the forgetful functor $\rep R\Ga\to\modules R$ and  $x^*:\modules{\cd(X/S)}\to\modules R$ are $\ot$-isomorphic.

\begin{cor}\label{23.02.2017--6}If $\Ga$ has an infinite cyclic quotient, then the group scheme $\cn$ of Section \ref{23.02.2017--7} appears as   the full differential Galois group of some $(\ce,\na)\in\modules{\cd(X/S)}^\circ$.  
\end{cor}
\begin{proof}Let $\Pi$ be the Tannakian  envelope of $\Ga$ (Definition \ref{22.02.2017--3}). Using \cite[3.2.1(ii)]{DH14} and then Proposition \ref{13.09.2017--1} we produce a faithfully flat morphism $u:\Pi\to\cn$. Let now $E$ be a free $R$-module of finite rank and  $\rho:\GG_a\ti\GG_m\to\bb{GL}(E)$ a faithful representation. We have a commutative diagram in $(\bb{FGSch}/R)$
\[
\xymatrix{
\Pi\ar[r]\ar[d]_{u}&\bb{GL}(E)\\
\cn\ar[r]&\GG_a\ti\GG_m\ar[u]_\rho;
}
\]
from \cite[Proposition 4.10]{DHdS15} we know that the full subcategory $\langle E\rangle_\ot$ of $\rep R\Pi$ is equivalent to $\rep R\cn$ under $u$. Now we translate these observations to the category $\modules{\cd(X/S)}$. 
\end{proof}

\section{Prudence of flat group schemes and projectivity of the underlying Hopf algebra}\label{07.07.2017--1}

In this section,   $R$ is assumed to be a  
\emph{complete} DVR. In what follows,   $G$ is a flat group scheme over $R$, which is \emph{not} assumed to be of  finite type.  
We shall investigate under which conditions $R[G]$  is a projective $R$-module. Our method is based on the following fundamental result. 

\begin{thm}[{cf. \cite[page 338]{kaplansky}}] \label{kaplansky}Let $M$ be a flat $R$-module whose associated $K$-module $M\ot K$ has countable rank. Then $M$ is free if and only if $\cap\pi^nM=(0)$. \qed
\end{thm}

Therefore, all efforts should be concentrated on understanding the \emph{divisible ideal}, which we now introduce. 

\begin{dfn}\label{divisible_ideal}Let $M$ be an $R$-module. We denote by $\g d(M)$ the submodule  $\cap \pi^nM$, and call it the divisible submodule of $M$. 
If $A$ is an $R$-algebra, $\g d(A)$ is an ideal, which we call the divisible ideal of $A$. 
If $X=\mathrm{Spec}\,A$, we write alternatively $\g d(X)$ to mean $\g d(A)$ and let $\mm{Haus}(X)$ stand for the closed subscheme of $X$ defined by $\g d(X)$. 
\end{dfn}

It is quite easy to see that  $\g d(G)$ cuts out a closed subgroup scheme (see Proposition \ref{15.09.2017--2}). This being so,  in order to study how far $\g d(G)$ is from being $(0)$, the natural idea is, following Chevalley, to search semi-invariants. This is the crucial idea behind our notion of \emph{prudence}, see Definition \ref{17.10.2016--4}. 

Section \ref{28.09.2017--1} explains what we understand by semi-invariants in the context of group schemes over $R$. In it we also derive basic results on this concept which prove useful further ahead. In Section \ref{28.09.2017--2}   we introduce the notion of a prudent group scheme; this property is compared to the vanishing of the divisible ideal in Section \ref{15.09.2017--1}. The function of Section \ref{28.09.2017--3} is to relate prudence to other existing properties (appearing in \cite{SGA3}, \cite{raynaud} and \cite{DH14}).
Finally, Section \ref{28.09.2017--4}  links prudence of differential Galois groups to Grothendieck's existence theorem of projective geometry.

\subsection{Generalities on semi-invariants}\label{28.09.2017--1}
Let  $V$ be an object of  $\rep{R}{G}$; comultiplication $V\to V\ot_RR[G]$ is denoted by $\rho$. 

\begin{dfn}\label{semi-invariant}
We say that $\bb v\in V$ is a semi-invariant if $R \bb  v$ is a subrepresentation \cite[Part I, 2.9]{jantzen}. The set of all semi-invariants in $V$ is denoted by ${\rm SI}_G(V)$.  
\end{dfn}
We need the following basic results on semi-invariants. 

\begin{lem}\label{17.10.2016--1}The element $\bb v\in V$ is a semi-invariant if and only if $\rho(\bb v)=\bb v\ot \te$ for some $\te\in R[G]$. 
\end{lem}
\begin{proof}If $R\bb v$ is a subrepresentation, then $\rho(\bb v)$ belongs to the $R$-submodule $R\bb v\ot_RR[G]$ of $V\ot R[G]$, and hence $\rho(\bb v)=\bb v\ot \te$ for some $\te\in R[G]$. On the other hand, if $\rho(\bb v)=\bb v\ot\te$ for some $\te\in R[G]$, then $\rho|_{R\bb v}$ factors through $R\bb v\ot R[G]$, which amounts to being a  subrepresentation \cite[Part I, 2.9]{jantzen}. 
\end{proof}

\begin{rmks}We refrain from calling $\te$ a group-like element because we \emph{do not} assume that $R\bb v$ is isomorphic to $R$. 
Note that [SGA 3, Expos\'e ${\rm VI}_B$] only uses the notion of semi-invariant   in the case of a base field, see Definition 11.15 on p. 420 in there.

The set of semi-invariants is not usually stable under addition, so that it cannot be an $R$-submodule. On the other hand, it is stable under scalar multiplication (apply Lemma \ref{17.10.2016--1}).
\end{rmks}

\begin{lem}\label{17.10.2016--2}The following claims are true. 
\begin{enumerate} 
\item Let $\ph:V\to W$ be a morphism of representations of $G$.  The image of $\mm{SI}_G(V)$ lies in $\mm{SI}_G(W)$.
\item Let $H$ be another flat affine group scheme over $R$ and $f:G\to H$ a morphism. Then, for any $V\in\rep{R}{H}$, we have $\mm{SI}_H(V)\subset\mm{SI}_G(V)$.  
\item Under the notations of the previous item, assume that $f^\#:R[H]\to R[G]$ is injective and that $V$ belongs to $\repp RH$. Then   $\mm{SI}_H(V)=\mm{SI}_G(V)$.
\end{enumerate}\end{lem}
\begin{proof}
The verification of (1) and (2) follows the obvious method and we omit it. 

(3) Let $\bb v\in V$ be $G$-semi-invariant. We assume, to begin with, that $\bb v$ is the first element of an ordered basis $\{\bb v_1,\ldots,\bb v_r\}$ of $V$. Write $\si:V\to V\ot R[H]$ for the coaction and let $\si(\bb v_j)=\sum_i\bb v_i\ot \si_{ij}$ with $\si_{ij}\in R[H]$. Our hypothesis jointly with Lemma \ref{17.10.2016--1} then implies that 
\[\sum_i\bb v_i\ot f^\#(\si_{i1}) =\bb  v_1\ot \te\]
for some $\te\in R[G]$. This shows that $f^\#(\si_{i1})=0$ if $i>1$ and, consequently, that $\si_{i1}=0$ if $i>1$. This implies that $\bb v$ is $H$-semi-invariant (again by Lemma \ref{17.10.2016--1}).

Now, for a general $\bb v\in V\setminus\{0\}$, there exists an ordered basis $\{\bb v_1,\ldots,\bb v_r\}$ and some non-negative integer $m$ such that $\pi^m\bb v_1=\bb v$. We then see that $\bb v_1$ is also $G$-semi-invariant: There exists $\te\in R[G]$ such that $(\id_V\ot f^\#)\circ\si(\pi^m\bb v_1)=\pi^m\bb v_1\ot\te$, and hence $(\id_V\ot f^\#)\circ\si(\bb v_1)=\bb v_1\ot\te$ as $\pi^m$ is not a zero-divisor in $V\ot R[G]$;  Lemma \ref{17.10.2016--1} proves that $\bb v_1$ is semi-invariant for $G$. From the previous considerations, $\bb v_1$ is  $H$-semi-invariant, so that $\bb v=\pi^m\bb v_1$ is also $H$-semi-invariant.\end{proof}

\subsection{The concept of prudence}\label{28.09.2017--2}

We are now able to isolate the property of $\rep{R}{G}$ which forces the ideal $\g d(G)$ to vanish.

\begin{dfn}\label{17.10.2016--4}
Let $V\in{\rm Rep}_R^{\circ}(G)$. We say that $V$ is prudent if  an element $\bb v$ whose reductions modulo $\pi^{n+1}V$ are semi-invariant for all $n$ must necessarily be semi-invariant. We say that $G$ is prudent if each $V\in{\rm Rep}_R^{\circ}(G)$ is prudent.  
\end{dfn}

The definition of prudence admits the following reformulation (the proof is similar to the proof of Lemma \ref{17.10.2016--2}(3), and we omit it). 

\begin{lem}\label{17.01.2018--1}Let $V\in\repp{R}{G}$. Assume that each $\bb v\in V$ which \begin{itemize}\item  belongs to a basis, and
\item has a  reduction modulo $\pi^{n+1}V$ which is semi-invariant for all $n$, 
\end{itemize}
is a semi-invariant.  
Then $V$ is prudent. \qed
\end{lem}


The following result shall be employed in deriving deeper consequences in Section \ref{15.09.2017--1}.  

\begin{lem}\label{17.10.2016--5}Let $f:G\to H$ be a  morphism in $(\bb{FGSch}/R)$ such that the associated morphism of Hopf algebras $f^\#:R[H]\to R[G]$ is injective.  Then, if $G$ is prudent, so is $H$. 
\end{lem}

\begin{proof}
Let $V\in\repp{R}{H}$ and consider an element $\bb v\in V$ whose reduction modulo $\pi^{n+1}V$ is $H$-semi-invariant for any given $n$. From Lemma \ref{17.10.2016--2}-(2), we have 
\[
\mm{SI}_H(V/\pi^{n+1})\subset \mm{SI}_G(V/\pi^{n+1}).
\]
This forces, as $G$ is prudent,  $\bb v$ to be $G$-semi-invariant. But by Lemma \ref{17.10.2016--2}-(3), we know that $\bb v$ is then $H$-semi-invariant. 
\end{proof}

\begin{ex}\label{03.10.2018--1}Let $\Ga$ be an abstract group with Tannakian envelope $\Pi$ (see Section \ref{10.02.2017--1} for definitions). 
We contend that  $\Pi$  is prudent. Let $V$ be a free $R$-module of   rank $r$ affording a linear representation of $\Ga$. Let $\bb v\in V$ be such that, for any given $n\in\NN$, the element $\bb v+\pi^{n+1}V$ is semi-invariant. This implies that $R\bb v+\pi^{n+1}V$, regarded as an $R$-submodule of $V$, is invariant under $\Ga$. Hence, $\cap_n(R\bb v+\pi^{n+1}V)$ is invariant under $\Ga$. But this intersection is the closure of $R\bb v$ for the $\pi$-adic topology \cite[pp 55-6]{matsumura}, which is simply $R\bb v$ \cite[Theorem 8.10]{matsumura}. 
\end{ex}

%
%
%
%

\subsection{Prudence and the divisible ideal}\label{15.09.2017--1}
Our objective here is to establish:  
\begin{thm}\label{17.10.2016--6}If $G$ is  prudent, then $\g d(  G)=0$.  
\end{thm}

Before embarking on a proof, we note a fundamental fact.   
\begin{prp}\label{15.09.2017--2}
The closed subscheme of $G$ defined by $\g d(G)$, $\mathrm{Haus}(G)$, is in fact a closed and flat subgroup scheme. 
\end{prp}

The proof of this Proposition follows easily  from item (3) of 
\begin{lem}Let $M$ and $N$ be flat $R$-modules.
\begin{enumerate}\item Assume that  $\g d(M)=(0)$, and let $E\subset M$ be finitely generated. Then $E^{\rm sat}$ is free.  
\item If   $\g d(M)=(0)$ and $\g d(N)=(0)$,  then $\g d(M\ot_R N)=(0)$. 
\item The divisible submodule of $M\ot N$ is simply $\g d(M)\ot N+M\ot\g d(N)$. 
\end{enumerate}      
\end{lem}
\begin{proof}(1) Clearly  $E\ot K\stackrel{\sim}{\to} E^{\rm sat}\ot K$, and we can apply  Theorem \ref{kaplansky} to  conclude. 

(2) Let $t=\sum u_i\ot v_i\in   M\ot N$ belong to $\g d(M\ot N)$. By saturating $\sum Ru_i$ and $\sum Rv_i$, and applying (1), we see that  there exist a free and saturated submodule $U$ of $M$ and a free and saturated submodule $V$ of $N$   such that $t\in U\ot V$. Now $U\ot V$ is saturated in $M\ot N$, and hence $\g d(M\ot N)\cap (U\ot V)=\g d(U\ot V)$. Therefore, $t=0$ since $\g d(U\ot V)=(0)$.

(3) This is a simple consequence of (2).
\end{proof}

We now head towards the proof of Theorem \ref{17.10.2016--6}. 
In fact, this theorem is a consequence of the ensuing result, which therefore becomes the object of our efforts. To state it, we employ the notion of the  diptych \cite[Section 4]{DHdS15}, briefly  recalled  here for the convenience of the reader.  Given an arrow $u:\ch\to\cg$ of $(\bb{FGSch}/R)$, there are two reasonable ways to define its image: the closed subgroup scheme of $\cg$ cut out by the kernel of $u^\#:R[\cg]\to R[\ch]$ (the schematic image of $u$) and the spectrum of the saturation of $u^\#(R[\cg])$ in $R[\ch]$. If $\Ps$ stands for the former and $\Ps'$ for the latter, we then have a commutative diagram 
\[
\xymatrix{\ch\ar[r]^u\ar[d]_{\text{faithfully flat}}&\cg\\\Ps'\ar[r]&\Ps\ar[u]_{\text{closed immersion}}, \
}\] 
called the diptych. 

\begin{thm}\label{17.10.2016--7}Assume that $G$ is prudent and let  $V\in\repp{R}{G}$ be arbitrary. Consider the diptych of $\rho$
\[
\xymatrix{\ar[d]_{\text{faithfully flat}}G\ar[r]^-{\rho}& \bb{GL}(V) \\ \Ps'\ar[r]&\Ps,\ar[u]_{\text{closed immersion}}}
\]
so that $\Psi'\ot K=\Ps\ot K$ is of finite type over $K$. Then $\g d(\Ps')=0$.  
\end{thm}

\begin{proof}[Proof that Theorem \ref{17.10.2016--7} $\Rightarrow$ Theorem \ref{17.10.2016--6}] Let $a\in\g d(G)$.  Endowing $R[G]$ with its right-regular action and using ``local finiteness'' \cite[Part I, 2.13]{jantzen}, we can find a subcomodule $V\subset R[G]$ which contains $a$ and is finitely generated as an $R$-module. Since $R[\Psi]$ is constructed as the image of the obvious morphism $R[\mathbf{GL}(V)]\to R[G]$, and $R[\Psi']$ is the saturation of $R[\Psi]$ in $R[G]$ (see \cite[Section 4]{DHdS15}), we conclude that $a\in R[\Psi']$. Now, because $R[\Psi']$ is saturated in $R[G]$, we have $\g d(\Psi')=\g d(G)\cap R[\Psi']$. If Theorem \ref{17.10.2016--7} is true, then $a=0$.  
\end{proof}

We need some lemmas to establish Theorem \ref{17.10.2016--7}. In these results we employ the notion of closure of a closed subscheme of  the generic fibre \ega{IV}{2}{2.8, 33ff} as well as the constructions of Neron blowups appearing in Section \ref{23.02.2017--3}.

\begin{lem}\label{17.10.2016--8}
Let $f:\cg'\to \cg$ be a morphism of $(\bb{FGSch}/R)$. We suppose that $\cg$ is of finite type and that $f\ot K$ is an isomorphism.  Let $H\subset \cg$ stand for the closure of \[{\rm Haus}(\cg')\ot K\subset \cg'\ot K=\cg\ot K\] in $\cg$. Then $f$ factors through  
\[
\cn^\infty_H(\cg)\aro \cg. 
\]
\end{lem}
\begin{proof}Algebraically, $H$ is cut out by the ideal 
\[
I=\{ a\in R[\cg]\,:\, 1\ot a\in K\ot \g d(\cg')  \}. 
\]
By definition, $R[\cn_H^\infty(\cg)]$ is the $R$-subalgebra of $K[\cg]$ obtained by adjoining to $R[\cg]$ all elements of the form $\pi^{-n}\ot a$, where $a\in I$. We then need to prove that for each $a\in I$ and each $n\in\NN$, there exists $\de_n\in R[\cg']$ such that $a=\pi^n\de_n$. 
So let $a\in I$. By construction,    $1\ot a$ belongs to the ideal $K\ot \g d(\cg')$. Hence, there exists some $m\in\NN$ and some $\de\in\g d(\cg')$ such that  $1\ot a=\pi^{-m}\ot \de$. This  shows that $\pi^m a\in\g d(\cg')$. Now, for every $n\in\NN$, we pick $\de_n$ such that $\pi^{n+m}\de_n=\pi^ma$, so that $\pi^n\de_n=a$. 

\end{proof}

\begin{lem}\label{17.10.2016--9}Let $\ch\to \cg$ be a closed embedding of flat   group schemes of finite type over $R$. If   $\cn_\ch^\infty(\cg)$ is prudent, then   $\ch=\cg$. 
\end{lem}
\begin{proof} 
From Lemma \ref{02.03.2017--1} below,  we can find  $V\in\repp{R}{\cg}$ and a line $R\bb v\subset V$
such that 
\[
\bb{Stab}_{\cg}(R\bb v)=\ch.
\]
The $\ch$-semi-invariance of $\bb v$ forces  $\bb v+\pi^{n+1}V\in V/\pi^{n+1}$ to be $\ch$-semi-invariant for all $n$. (Use Lemma \ref{17.10.2016--2}-(1).) But since for each $n$ the base change
\[\cn^\infty_\ch(\cg)\ot R_n\aro \cg\ot R_n\]
factors through a morphism 
\[
\cn_\ch^\infty(\cg)\ot R_n\aro \ch\ot R_n,
\]
and this is the central point, we conclude that  $\bb v+\pi^{n+1}V$ is $\cn_\ch^\infty(\cg)$-semi-invariant. (The arrow $\cn_\ch^\infty(\cg)\ot R_n\to \ch\ot R_n$ is even an isomorphism \cite[Corollary 5.11]{DHdS15}.) As $\cn_\ch^\infty(\cg)$ is assumed prudent, $\bb v$ is $\cn_\ch^\infty(\cg)$-semi-invariant as well. But then,  Lemma \ref{17.10.2016--2}-(3) guarantees that $\bb v\in\mathrm{SI}_{\cg}(V)$. This implies that $R\bb v$ is stable under $\cg$, so that $\cg=\ch$.
\end{proof}

The next result is widespread in the case of a base-field, see for example \cite[16.1]{waterhouse}. It is explicitly proved in \cite[Proposition 1.2]{pappas-rapoport08} but we include the details because the proof given in \cite[16.1]{waterhouse}---which we follow---is quite transparent. 

\begin{lem}\label{02.03.2017--1}Let $\ch\to \cg$ be a closed embedding in $(\bb{FGSch}/R)$. If  $\cg$ is of finite type, then, there exists a certain $V\in\repp{R}{\cg}$ and an element $\bb v\in V$ such that 
\[
\mathbf{Stab}_{\cg}(R\bb v)=\ch.
\] 
\end{lem}

\begin{proof}
Let $V\subset R[\cg]$ be a 
subcomodule of the right regular representation which, as an $R$-module, is of finite type and contains generators of the ideal $I$ cutting out $\ch$.   Let $W:=V\cap I$ and note that $W$ is saturated in $V$  since $I\subset R[\cg]$ is. 
Let then $\{\bb v_1,\ldots,\bb v_s\}$ be a basis of $V$ over $R$ such that $\{\bb v_1,\ldots,\bb v_r\}$ is a basis of $W$; in this case $R[\cg]\bb v_1+\cdots+R[\cg]\bb v_r=I$. Write $\De \bb v_j=\sum \bb v_i\ot a_{ij}$, so that $\bb{Stab}_\cg(W)$ is cut out by the ideal $J=(a_{ij}\,:\,i>r,\,\,  j\le r)$, see the  proof of the last Lemma in \cite[12.4]{waterhouse}. We now claim that $J=I$, hence proving that $\ch=\bb{Stab}_{\cg}(W)$. 

First, since $\bb v_j=(\ep\ot\id)\circ \De\bb v_j$, we see, using that $\ep(\bb v_1)=\ldots=\ep(\bb v_r)=0$, that $\bb v_j=\sum_{i=r+1}^s\ep(\bb v_i)a_{ij}$, which shows that $I\subset J$. Now, let $\ps:R[\cg]\to R[\ch]$ be the projection and consider the $R[\ch]$-comodule defined by $(\id\ot\ps)\De:R[\cg]\to R[\cg]\ot R[\ch]$. (This corresponds to the action of $\ch$ on $\cg$ by multiplication on the right.) Then, we know that $V$ and $I$ are $R[\ch]-$subcomodules, so that $(\id\ot\ps)\De(I)\subset I\ot R[\ch]$ and $(\id\ot\ps)\De(V)\subset V\ot R[\ch]$. This proves that $(\id\ot\ps)\De(W)\subset W\ot R[\ch]$, because $R[\ch]$ is $R$-flat \cite[Theorem 7.4(i)]{matsumura}.  Therefore, since $\bb v_1,\ldots,\bb v_r\in W$, we see that $(\id\ot\ps)\De(\bb v_j)\in W\ot R[\ch]$ provided that $1\le j\le r$. But $(\id\ot\ps)\De(\bb v_j)=\sum_{i\le r}\bb v_i\ot\ps(a_{ij})+\sum_{i>r}\bb v_i\ot\ps(a_{ij})$, and hence $\ps(a_{ij})=0$, if $i>r$ and $j\le r$. This proves the inclusion $J\subset I$, and our claim. 

The fact that $W$ can be chosen of rank one follows from $\bb{Stab}_{\cg}(W)=\bb{Stab}_\cg(\wedge^s W)$ \cite[Appendix 2]{waterhouse}.
\end{proof}

\begin{proof}[Proof of Theorem \ref{17.10.2016--7}]Since $G$ is prudent, Lemma \ref{17.10.2016--5} guarantees that $\Ps'$ is prudent. Let $H\subset \Psi$ be the closure of ${\rm Haus}(\Ps')\ot K\subset \Psi\ot K$ in  $\Ps$.  Lemma \ref{17.10.2016--8} shows that $\Psi'\to \Psi$ factors as 
\[
\Psi'\aro \cn_{H}^\infty(\Psi)\aro\Psi.
\]
Another application of Lemma \ref{17.10.2016--5} ensures that $\cn^\infty_H(\Psi)$ is prudent because $\Psi'$ is. As $\Ps$ is of finite type, Lemma \ref{17.10.2016--9} tells us that $H=\Psi$. Hence,  the closed embedding ${\rm Haus}(\Ps')\subset\Ps'$ becomes an isomorphism when base-changed to $K$. This is only possible if ${\rm Haus}(\Ps')=\Ps'$, and we are done. 
\end{proof}

As a simple application of Theorem \ref{17.10.2016--6}, we note that if $\cn$ is the group scheme studied in Section \ref{23.02.2017--7}, then $R[\cn]$ is a free $R$-module. To see this, let $\Pi$ stand for the Tannakian envelope of the infinite cyclic group $\Ga$ and $u_\ph:\Pi\to\cn$ for the \emph{faithfully flat} morphism constructed in Proposition \ref{13.09.2017--1}. Since   $\Pi$ is a prudent group scheme (see Example \ref{03.10.2018--1}) we conclude that  $(0)=\g d(\Pi)=\g d(\cn)$, which implies, due to Theorem \ref{kaplansky}, freeness of $R[\cn]$.

\subsection{Prudence and some other characterizations}\label{28.09.2017--3}
Fundamental as it is, Kaplansky's theorem has the drawback of needing the hypothesis on the rank.  Using Raynaud-Gruson's notion of Mittag-Leffler modules \cite[Ch. 2, \S2, Definition 3]{raynaud} (see also \cite[Part 2]{raynaud-gruson}), we can put our knowledge in a more efficient structure and examine two 
other distinctive features which were singled out in  \cite{DH14} (repeated in Definition \ref{specially_locally_finite} below) and in \cite[${\rm VI}_B$, 11]{SGA3}.
We recall that $G$ is a flat group scheme over the complete discrete valuation ring $R$. 



\begin{dfn}[{cf. \cite[Section 3.1]{DH14}}]\label{specially_locally_finite} A $G$-module $M$ is called \emph{specially locally finite} if for each $G$-submodule $V\subset M$ whose underlying $R$-module is of finite type, the saturation $V^{\rm sat}\subset M$ is also of finite type (over $R$). The group scheme $G$ is specially locally finite if $R[G]_{\rm right}$  is specially locally finite. 
\end{dfn}

As remarked already in \cite[Proposition 3.1.5]{DH14}, the above property is closely related to the Mittag-Leffler condition. The ensuing result straightens this relation.

\begin{prp}\label{16.10.2018--2}The following conditions   are equivalent. 
\begin{enumerate}\item The group scheme $G$ is prudent. 
\item The divisible ideal $\g d(G)$ is null. 
\item The group scheme $G$ is specially locally finite. 
\item The $R$-module $R[G]$ satisfies the Mittag-Leffler condition of Gruson-Raynaud. 
\item For each $V\in\repp{R}{G}$ and each finite subset $F\subset V$, the set 
\[
\{\text{$W$ subrepresentation of $V$ containing $F$}\}\]
has a least element. 
\end{enumerate}
\end{prp}
\begin{proof}
(1) $\Rightarrow$ (2). This is the content of   Theorem \ref{17.10.2016--6}.

(2) $\Rightarrow$ (3). We know from \cite[Theorem 4.1.1]{DH14} that $G$ is the limit in $(\bb{FGSch}/R)$ of a system \[\{G_\la,\ph_{\la,\mu}:G_\mu\to G_\la\}\] where: 
 (a) each $G_\la\ot K$ is of finite type, and 
(b) each $\ph_{\la\mu}$ is faithfully flat. 
Since $\g d(G)=0$, we can certainly say that  $\g d(G_\la)=0$ for all $\la$, so that, by 
Kaplansky's Theorem (Theorem \ref{kaplansky}),  $R[G_\la]$ is free as an $R$-module. Now, $R[G]_{\rm right}$ is the direct limit of a system of $G$-modules,  $R[G]_{\rm right}=\lid_\la R[G_\la]$, and, by Proposition 3.1.5-(ii) of \cite{DH14},  $R[G_\la]$ is specially locally finite. As each $R[G_\la]$ is saturated in $R[G]$, it becomes a  simple matter to show that       $R[G]$ is also specially locally finite.

(3) $\Rightarrow$ (4). This is already explained in \cite[Proposition 3.1.5-(i)]{DH14}.


(4) $\Rightarrow$ (5). We write $G=\lip_\la  G_\la$ as in the proof of ``(2) $\Rightarrow$ (3).'' 
Since the canonical morphism $G\to G_\la$ is faithfully flat, condition (4) and Corollary 2.1.6 of  \cite[Part 2, \S1]{raynaud-gruson} guarantee  that  $R[G_\la]$ satisfies the Mittag-Leffler condition. The fact that $K[G_\la]$ is of finite type jointly with Proposition 1 of \cite[Chapter 2, \S2]{raynaud} assure in turn that, for each $\la$, $R[G_\la]$ is a projective $R$-module.

Let now $V$ and $F$ be as in condition (5). It is possible to find some $\la$ such that $V$ actually comes from a representation of $G_\la$. Now   \cite[${\rm VI}_B$, 11.8.1, p.418]{SGA3} assures the existence of a $G_\la$-subrepresentation $W$ containing $F$ and not properly containing any other such $G_\la$-subrepresentation. Since any $G$-subrepresentation of $V$ must actually come from a representation of $G_\la$ \cite[Theorem 4.1.2(i)]{DH14}, we are done.    

(5) $\Rightarrow$ (1). Let $V\in\repp{R}{G}$. If $\bb v\in V$ is such that its image in $V/\pi^{n+1}$ is a semi-invariant, it follows that $W_n:=R\bb v+\pi^{n+1}V$ is a subrepresentation. Now, let $W\subset V$ be a  subrepresentation containing $\bb v$ and not properly containing any other such. It then follows that $R\bb v\subset W\subset\cap W_n=R\bb v$, so that $R\bb v$ is a subrepresentation. By definition, $\bb v$ is semi-invariant, and $G$ must be prudent.  
\end{proof}

\subsection{Applications to the full differential Galois group}\label{28.09.2017--4}

Let $f:X\to S$ be a smooth and proper morphism having geometrically connected fibres and admitting a section $\xi\in X(R)$. Given $\ce\in\dmod{X/R}^\circ$, let $\mm{Gal}'(\ce)$ be the full differential Galois group. (See   Section 7 of \cite{DHdS15} or   Section \ref{22.08.2020--1}.)

\begin{thm}\label{prudence_galois}The ring of functions of the group $\mathrm{Gal}'(\ce)$ is a free $R$-module. 
\end{thm}

\begin{proof}For brevity, we write $G'=\mm{Gal}'(\ce)$ and $E=\xi^*\ce$. We need to show that $G'$ is prudent to apply Theorem \ref{17.10.2016--6} and Theorem \ref{kaplansky}. So let $\cv\in\modules{\cd(X/S)}^\circ$ belong to $\langle\ce\rangle_\ot$ and write $V=\xi^*\cv$; this free $R$-module affords a representation of $G'$. 
Let now $\bb v$ be a non-zero element of $V$ whose image in  $V_n=V/\pi^{n+1}$, call it  $\bb v_n$, 
is semi-invariant.  It is to prove that 
$R\bb v$ is also a subrepresentation of $V$. As Lemma \ref{17.01.2018--1} guarantees, there is no loss of generality in supposing that $\bb v$ belongs to a basis of $V$.

We denote by $\ps_n:L_n\to V_n$ the inclusion of $R\bb v_n$ (an $R$-module isomorphic to $R_n$) into $V_n$.
The equivalence $\xi^*:\langle  \ce \rangle_\ot\stackrel{\sim}{\to}\langle E\rangle_\ot$ produces, from the commutative diagram with exact rows 
\[
\xymatrix{
L_{n+1}\ar[r]^-{\pi^{n+1}}\ar@{^{(}->}[d]_{\ps_{n+1}} & L_{n+1}\ar@{^{(}->}[d]^{\ps_{n+1}}  \ar[r] &  \ar@{^{(}->}[d]^{\ps_n}L_n\ar[r]&  0 \\ V_{n+1}\ar[r]^-{\pi^{n+1}} & V_{n+1}  \ar[r] &  V_n\ar[r]&  0,}
\]
a commutative diagram with exact rows in $\modules{\cd(X/S)}$: 
\[
\xymatrix{
\mathcal L_{n+1}\ar[r]^-{\pi^{n+1}}\ar@{^{(}->}[d]_{\phi_{n+1}} &\mathcal  L_{n+1}\ar@{^{(}->}[d]^{\phi_{n+1}}  \ar[r] &  \ar@{^{(}->}[d]^{\phi_n}\mathcal L_n\ar[r]&  0 \\ \mathcal V_{n+1}\ar[r]^-{\pi^{n+1}} & \mathcal V_{n+1}  \ar[r] &  \mathcal V_n\ar[r]&  0.}
\]

By GFGA \cite[Theorem 8.4.2]{Illusie}, there exists an arrow of coherent sheaves 
\[
\phi:\cl\aro\cv
\]
inducing $\phi_n$ for each $n$. Using Lemma \ref{14.09.2017--1} and the fact that each $\phi_n$ is a monomorphism we conclude that $\phi$ is also a monomorphism. Applying the claim made on the proof of essential surjectivity in Proposition \ref{03.07.2017--1}, we conclude that $\cl$ carries a structure of $\cd(X/S)$-module. In addition, $\phi$ is then an arrow of $\cd(X/S)$-modules, according to the claim made on the proof of fully faithfulness in Proposition \ref{03.07.2017--1}. As we already remarked that $\phi$ is a monomorphism, we conclude that  $\phi$ is really an arrow of $\langle\ce\rangle_\ot$.

 Let $\ps:L\to V$ be an arrow in $\rep{R}{G'}$ whose image under $\xi^*$ is $\phi$ so that $\ps$ induces $\ps_n$ upon reduction modulo $\pi^{n+1}$. It only takes a moment's thought to see that  $R\bb v=\ps(L)$, so that $R\bb v$ is a subrepresentation of $V$.
\end{proof}


The various differential Galois groups can be put together to form the relative fundamental group scheme of $X$ based at $\xi$, see \cite[Section 5.1]{DH14}. If we denote this group scheme by $\Pi$, then $\Pi=\lip_\la G_\la$, where each $R[G_\la]$ is a free $R$-module (by Theorem \ref{prudence_galois}) and the transition arrows $G_\mu\to G_\la$ are faithfully flat (we use \cite[Theorem 4.1.2(i)]{DH14}). Then    \cite[Tag 0AS7]{stacks} tells us that 

\begin{cor}The ring of functions of the relative  fundamental group scheme  (cf.  \cite[Section  5.1]{DH14}) of $X/R$ based at $\xi$  satisfies the Mittag-Leffler condition.  \qed
\end{cor}

Let us end this section by putting together Theorem \ref{prudence_galois} and the results on unipotent group schemes explained in Section \ref{06.02.2017--1}.

\begin{cor}\label{16.10.2018--1}Assume that $\ce_K$ is unipotent and that $R$ is of characteristic $(0,0)$. Then $\mathrm{Gal}'(\ce)$ is of finite type. 
\end{cor}

\begin{proof}We consider the morphism 
${\rm Gal}'(\ce)\to{\rm Gal}(\ce)$
from the full to the restricted differential Galois group (see Section 7 of \cite{DHdS15}).  Since $\ce_K$ is unipotent,  ${\rm Gal}(\ce)\ot K$ is unipotent and, a fortiori, connected \cite[IV.2.2.5, p.487]{DG}. 
Hence,    \cite[Theorem 2.11, p.187]{tong} shows that ${\rm Gal}(\ce)$ is 
unipotent over $R$. We   now  apply Corollary \ref{29.09.2017--1} and Theorem \ref{prudence_galois} to conclude that ${\rm Gal}'(\ce)$ is of finite type. 
\end{proof}

\begin{rmk}In Theorem \ref{prudence_galois}, suppose that $R$ is of characteristic $(0,0)$. From Corollary \ref{11.07.2017--2} we know that ${\rm Gal}'(\ce)$ is of the form $\cn^\infty_{\g H}(G)$ with $G$ of finite type over $R$. Theorem \ref{prudence_galois} then says that $\g H$ is {\it not} algebraizable, that is, it is not the completion of some closed flat subgroup $H\subset G$ (otherwise $R[{\rm Gal}'(\ce)]$ would contain a copy of $K$).   
\end{rmk}

\begin{rmk}Theorem \ref{prudence_galois} is certainly false in case $X$ is simply affine. See Example 7.11 in \cite{DHdS15}.
\end{rmk}

\end{document}